\documentclass[a4]{amsart}

\usepackage[utf8]{inputenc}
\usepackage[english]{babel}
\usepackage{amssymb}
\usepackage{mathtools}
\usepackage[mathscr]{euscript}
\usepackage{mathrsfs}
\usepackage{stmaryrd}
\usepackage{enumitem}
\usepackage[normalem]{ulem}
\usepackage{color}
\setenumerate[1]{label=(\alph*)}
\setenumerate[2]{label=(\roman*)}
\usepackage[all]{xy}
\usepackage{tikz}
\usetikzlibrary{matrix,arrows}
\usepackage{rotating}
\usepackage{extarrows}
\usepackage{accents}
\usepackage{leftidx}
\usepackage{upgreek}
\usepackage{bbm}
\usepackage{bbold}
\usepackage{xr-hyper}
\usepackage{hyperref}

\hypersetup{
  pdftitle    = {Smoothness of Derived Categories of Algebras},
  pdfauthor   = {Alexey Elagin, Valery A.~Lunts, Olaf
    M.~Schn{\"u}rer},
  colorlinks=true}

\newtheorem{mtheorem}{Theorem}
\newtheorem{theorem}{Theorem}[section]
\newtheorem{proposition}[theorem]{Proposition}
\newtheorem{lemma}[theorem]{Lemma}
\newtheorem{corollary}[theorem]{Corollary}

\theoremstyle{definition}
\newtheorem{definition}[theorem]{Definition}

\theoremstyle{remark}
\newtheorem{remark}[theorem]{Remark}

\DeclareMathOperator{\im}{Im}

\DeclareMathOperator{\Hom}{Hom}
\DeclareMathOperator{\End}{End}

\DeclareMathOperator{\Cone}{Cone}
\DeclareMathOperator{\colim}{colim}

\DeclareMathOperator{\nil}{nil}
\DeclareMathOperator{\Z}{Z}

\DeclareMathOperator{\Quot}{Quot}

\DeclareMathOperator{\Spec}{Spec}

\DeclareMathOperator{\HH}{H}
\DeclareMathOperator{\M}{M}

\newcommand{\kk}{{\mathsf{k}}}
\newcommand{\Yo}{\operatorname{Yo}}
\newcommand{\pseudocoh}{{\operatorname{ps-coh}}}
\newcommand{\hProj}{{\operatorname{hProj}}}
\newcommand{\hInj}{{\operatorname{hInj}}}
\newcommand{\per}{{\operatorname{per}}}

\newcommand{\res}{{\operatorname{res}}}

\newcommand{\mfp}{{\mathfrak{p}}}

\renewcommand{\epsilon}{{\varepsilon}}

\renewcommand{\phi}{{\varphi}}

\newcommand{\define}[1]{{\textbf{#1}}}

\newcommand{\id}{{\operatorname{id}}}
\newcommand{\opp}{{\operatorname{op}}}

\newcommand{\bZ}{{\mathbb{Z}}}
\newcommand{\bN}{{\mathbb{N}}}

\newcommand{\red}{{\operatorname{red}}}
\newcommand{\reg}{{\operatorname{reg}}}
\newcommand{\rad}{{\operatorname{rad}}}

\newcommand{\qc}{{\operatorname{qc}}}
\DeclareMathOperator{\coh}{coh}
\DeclareMathOperator{\Qcoh}{Qcoh}
\DeclareMathOperator{\Mod}{Mod}
\let\mod\relax
\DeclareMathOperator{\mod}{mod}
\DeclareMathOperator{\pdim}{pdim}

\newcommand{\pf}{{\operatorname{pf}}}
\newcommand{\bd}{{\operatorname{b}}}
\newcommand{\C}{\operatorname{{C}}}
\newcommand{\D}{{\operatorname{D}}}
\newcommand{\K}{{\operatorname{K}}}
\newcommand{\Rd}{\mathsf{R}}

\newcommand{\ol}[1]{{\overline{#1}}}

\newcommand{\ra}{\rightarrow}

\newcommand{\xra}[2][]{\xrightarrow[{#1}]{#2}}
\newcommand{\xla}[2][]{\xleftarrow[{#1}]{#2}}
\newcommand{\sira}{\xra{\sim}}

\newcommand{\sila}{\xla{\sim}}

\newcommand{\sra}{\twoheadrightarrow}
\newcommand{\hra}{\hookrightarrow}

\newcommand{\thick}{{\operatorname{thick}}}

\newcommand{\ul}[1]{{\underline{#1}}}

\newcommand{\sptag}[1]{\href{http://stacks.math.columbia.edu/tag/#1}{#1}}

\numberwithin{equation}{section}

\title{Smoothness of Derived Categories of Algebras}

\author{Alexey Elagin, Valery A.~Lunts, and Olaf M.~Schn{\"u}rer}

\address{A.E.: Institute for Information Transmission Problems (Kharkevich
  Institute), Russian Federation;
  HSE University, Russian Federation
}

\email{alexelagin@rambler.ru}

\address{V.L.:
  Department of Mathematics\\
  Indiana University\\
  831 East 3rd Street\\
  Bloomington, IN 47405\\
  USA;
  HSE University, Russian Federation
}
\email{vlunts@indiana.edu}

\address{O.S.:
  Institut f\"ur Mathematik\\
  Universit{\"a}t Paderborn\\
  Warburger Stra\ss{}e 100\\
  33098 Paderborn\\
  Germany
}

\email{olaf.schnuerer@math.uni-paderborn.de}

\begin{document}
\maketitle

\begin{abstract}
  We prove smoothness in the dg sense of the bounded derived
  category
  of finitely generated modules over any finite-dimensional
  algebra
  over a perfect field, thereby answering a question of Iyama.
  More generally, we prove this
  statement for any algebra over a perfect field that is
  finite over its center and whose center is finitely generated
  as an algebra. These results are deduced from a general sufficient
  criterion for smoothness.
\end{abstract}

\tableofcontents

\section{Introduction}
\label{sec:introduction}

Many triangulated categories have dg enhancements (differential
graded enhancements). If we consider
triangulated categories of algebraic or geometric origin, e.\,g.\
derived categories of modules over an algebra or of sheaves on some
space,
it is natural to ask what properties their dg enhancements have
and how these properties do or do not depend on properties of the
algebra or the space.
The focus of this article is on the smoothness
of dg enhancements (see Definition~\ref{d:smooth}), where we
always work over a field $\kk$.
Since dg enhancements are often essentially
unique
(see \cite{lunts-orlov-enhancement},
\cite{canonaco-stellari-uniqueness-of-dg-enhancements}), we are a
bit sloppy in this introduction and just say that a triangulated
category is smooth when we mean that a certain natural dg
enhancement has this property
(cf.\ Definition~\ref{d:T-smooth} and Remark~\ref{r:DbmodA-smooth}
for the choices used in this article).

For example, a quasi-projective scheme $X$ over a field $\kk$ is
smooth in the sense of algebraic geometry if and only if
the category
$\D_\pf(X)$ of perfect complexes on $X$ is smooth in
the dg sense (see
\cite{lunts-categorical-resolution,
  valery-olaf-new-enhancements}).
However, if the field $\kk$ is perfect, the
bounded derived category $\D^\bd(\coh(X))$ of coherent sheaves on
$X$ is always
smooth, regardless of $X$ being smooth or not
(see
\cite{lunts-categorical-resolution,
  valery-olaf-new-enhancements}). This example illustrates the
phenomenon that
different triangulated subcategories
of the unbounded derived category $\D_\qc(X)$
naturally associated to $X$
may or may not detect smoothness of $X$.

If $A$ is a noetherian $\kk$-algebra (associative and unital, but
not necessarily
commutative) it is natural and
interesting to ask whether the bounded derived category
$\D^\bd(\mod(A))$ of finitely generated $A$-modules is
smooth. (The category $\per(A)$ of
perfect complexes of $A$-modules, for an arbitrary $\kk$-algebra
$A$, is not so interesting: it is
smooth if and only if the
$A \otimes_\kk A^\opp$-module $A$ has finite projective
dimension.)
If $A$ is commutative and finitely generated, the answer is clear
from the above discussion by taking $X=\Spec A$. Hence one may
hope that $\D^\bd(\mod(A))$ is always smooth.

In this article, we extend the methods of
\cite{lunts-categorical-resolution,
  valery-olaf-new-enhancements} to prove
the smoothness of bounded derived
categories for some classes of noncommutative algebras.

\begin{mtheorem}
  [{see Theorem~\ref{t:Dbmod-findimalg-separable-smooth}}]
  \label{t:intro:Dbmod-findimalg-separable-smooth}
  Let $A$ be a finite-dimensional algebra over a field $\kk$ such
  that $\frac{A}{\rad(A)}$ is separable over $\kk$ (this
  condition is automatic if $\kk$ is perfect).
  Then
  $\D^\bd(\mod(A))$ is smooth over $\kk$.
\end{mtheorem}

This theorem answers affirmatively a question of Osamu Iyama
\cite{iyama-oberwolfach}. In down-to-earth terms it says that
the dg endomorphism algebra of a projective resolution of the
direct sum of the simple $A$-modules is perfect as a bimodule
over itself (see Remark~\ref{r:dgEndP-smooth}).

We also prove the following partial generalization of
Theorem~\ref{t:intro:Dbmod-findimalg-separable-smooth}.

\begin{mtheorem}
  [{see Theorem~\ref{t:DbmodA-smooth}}]
  \label{t:intro:DbmodA-smooth}
  Let $A$ be an algebra
  over a
  perfect field
  $\kk$. Assume
  that $A$ is a finite module over its center $\Z(A)$
  and that the center $\Z(A)$ is a finitely generated
  $\kk$-algebra.
  Then
  $\D^\bd(\mod(A))$ is smooth over $\kk$.
\end{mtheorem}

In fact, we prove a general result
from which both Theorems
\ref{t:intro:Dbmod-findimalg-separable-smooth}
and \ref{t:intro:DbmodA-smooth} follow: Given a $\kk$-algebra
$A$ (or, more generally, a $\kk$-linear category $A$) and a
triangulated subcategory $\mathcal{T}$ of $\D(A)$,
Theorem~\ref{t:sufficient-for-smoothness} gives a sufficient
condition for the smoothness of $\mathcal{T}$.
We expect that this theorem could be used for example to prove
the smoothness of $\D^\bd(\mod(A))$ for certain noetherian
$\kk$-algebras $A$.

Let us mention two results of independent interest used in the
proof of Theorem~\ref{t:intro:DbmodA-smooth}. The
first result concerns the existence of a classical generator of
the
bounded derived category of coherent modules over a noncommutative structure sheaf.

\begin{mtheorem}[{see Theorem~\ref{t:generator-DbcohA}}]
  \label{t:intro:generator-DbcohA}
  Let $X$ be a noetherian J-2 scheme (see
  Definition~\ref{d:j2})
  and
  $\mathcal{A}$ a coherent $\mathcal{O}_X$-algebra (which is
  assumed to be unital and
  associative, but not necessarily commutative). Then
  $\D^\bd(\coh(\mathcal{A}))$ has a classical generator.
\end{mtheorem}

The J-2 condition in this result is actually very natural
by the following interesting recent result
\cite[Prop.~2.8]{iyengar-takahashi-openness-2018}
by Iyengar and Takahashi:
A commutative noetherian ring $R$ is J-2 if and only if
$\D^\bd(\mod(A))$ has a classical generator for any finite
commutative $R$-algebra $A$.

The proof of Theorem~\ref{t:intro:generator-DbcohA}
is based on a Verdier localization
sequence given by
the following theorem and a technical
result using Azumaya
algebras (see
Proposition~\ref{p:locally-nilpotently-azumaya}).

\begin{mtheorem}[{see Theorem~\ref{t:verdier-open-closed-cohA}}]
  \label{t:intro:verdier-open-closed-cohA}
  Let $X$ be a
  noetherian scheme and
  $\mathcal{A}$ a
  coherent $\mathcal{O}_X$-algebra.
  Let $U$ be an open subscheme of $X$ and $Z :=X-U$ its
  closed complement.
  Then the sequence of triangulated categories
  \begin{equation*}
    \D^\bd_Z(\coh(\mathcal{A}))
    \ra
    \D^\bd(\coh(\mathcal{A}))
    \ra
    \D^\bd(\coh(\mathcal{A}|_U))
  \end{equation*}
  is a Verdier localization sequence
  (see Definition~\ref{d:verdier-sequence})
  where the first arrow is the
  inclusion and the second arrow is restriction to $U$.
\end{mtheorem}

\subsection{Acknowledgments}
\label{sec:acknowledgments}

It is our pleasure to thank Darrell Haile for some useful
conversations and Amnon Neeman for his help with
Theorem~\ref{t:verdier-open-closed-cohA}.
We thank Srikanth Iyengar and Ryo Takahashi for making us aware
of their article \cite{iyengar-takahashi-openness-2018}.
We thank Osamu Iyama and Steffen Oppermann for
suggesting that
Theorem~\ref{t:intro:Dbmod-findimalg-separable-smooth}
may generalize to non-positive dg algebras
with finite-dimensional cohomology. We thank the referee for
useful comments.

Valery Lunts was partially supported by Laboratory of Mirror
Symmetry NRU HSE, RF Government grant, ag.\ No.\ 14.641.31.0001.
Alexey Elagin's study has been funded within the framework of the HSE University Basic Research Program and the Russian Academic Excellence Project '5-100'.

\subsection{Conventions}
\label{sec:conventions}

We fix a field $\kk$.
Whenever $\kk$ is present, all categories and functors are
$\kk$-linear.
By a dg category we mean a $\kk$-linear dg category.
Sometimes we write $\otimes$ instead of $\otimes_\kk$ or
$\otimes_{\mathcal{O}_X}$.

Rings and algebras are assumed to be associative and unital, but
not necessarily commutative.
An algebra over a commutative ring $R$ is a
ring $A$ together with a morphism $R \ra A$ of rings landing in
the center $\Z(A)$ of $A$. Similarly, if $(X, \mathcal{O}_X)$ is
a ringed space where $\mathcal{O}_X$ is a sheaf of commutative
rings,
an $\mathcal{O}_X$-algebra is a sheaf
$\mathcal{A}$ of rings together with a morphism $\mathcal{O}_X
\ra \mathcal{A}$ of sheaves of
rings landing in
the center of $\mathcal{A}$.

By a module we mean a right module.
If $R$ is a ring, $\Mod(R)$ denotes the category of $R$-modules
and $\D(R)$ its derived category.
When we say that a ring is \emph{noetherian} we
mean that it is \emph{right noetherian}.
If $R$ is a noetherian ring,
$\mod(R)$ denotes the full subcategory of $\Mod(R)$ of finitely
generated $R$-modules, $\D(\mod(R))$ its derived category and
$\D^\bd(\mod(R))$ its subcategory of objects with bounded
cohomology.
If $R$ is a finite-dimensional algebra over a field $\kk$, then
$\mod(R)$ is just the category of finite-dimensional $R$-modules.

A thick subcategory of a triangulated category $\mathcal{T}$ is a
strictly full triangulated subcategory that is closed under
taking direct summands in $\mathcal{T}$. Given an object $E$ of
$\mathcal{T}$ we write $\thick(E)=\thick_\mathcal{T}(E)$ for the
smallest thick
subcategory of $\mathcal{T}$ containing $E$. The object $E$ is a
classical generator of $\mathcal{T}$ if and only if
$\thick(E)=\mathcal{T}$.

\section{Smoothness of derived categories of linear categories}
\label{sec:smoothn-deriv-categ}

This section is written in greater generality than needed in the
rest of this article.
The main result of this section,
Theorem~\ref{t:sufficient-for-smoothness},
concerns categories of modules over $\kk$-linear categories;
we apply this theorem later on only to categories of modules
over $\kk$-algebras.

\subsection{Modules over \texorpdfstring{$\kk$}{k}-linear
  categories}
\label{sec:modul-over-texorpdfs}

Let $\mathcal{A}$ be a $\kk$-linear category.
The reader may think of a $\kk$-algebra which is the same thing
as a $\kk$-linear category with precisely one object. In the rest
of this article, the results of this section will only be applied
in this special case.

A (right) $\mathcal{A}$-module is a functor
$\mathcal{A}^\opp
\ra \Mod(\kk)$, where $\Mod(\kk)$ is the category of $\kk$-modules.
Let $\Mod(\mathcal{A})$ be the (abelian) category of
$\mathcal{A}$-modules. Let
\begin{align*}
  \Yo \colon \mathcal{A} & \ra \Mod(\mathcal{A}),\\
  A & \mapsto \mathcal{A}(-,A),
\end{align*}
be the Yoneda functor. The objects of the essential image of this
functor are the representable $\mathcal{A}$-modules.
An $\mathcal{A}$-module is
finitely generated if it is a
quotient of a finite coproduct of representable
$\mathcal{A}$-modules. It is
free
if it is isomorphic to a coproduct of representable
$\mathcal{A}$-modules. It follows that
a finitely generated free $\mathcal{A}$-module is isomorphic to a
finite coproduct of representable $\mathcal{A}$-modules.

Let $\ul\C(\mathcal{A}):=\ul\C(\Mod(\mathcal{A}))$ be the dg
category of complexes of $\mathcal{A}$-modules.
Let $\C(\mathcal{A})$ be the category with the same
objects
whose morphisms are the closed degree zero morphisms in
$\ul\C(\mathcal{A})$. Let $\D(\mathcal{A})$ be the derived category of
$\mathcal{A}$-modules.

Since $\D(\mathcal{A})$ has arbitrary coproducts, it is
Karoubian, i.\,e.\ idempotent complete. Therefore, a
strictly full subcategory of
$\D(\mathcal{A})$ is Karoubian if and only if it is closed under
taking direct summands in $\D(\mathcal{A})$.

We say that an object of $\D(\mathcal{A})$ is
\define{pseudo-coherent}
if it is isomorphic to a bounded above complex of finitely
generated free $\mathcal{A}$-modules
(cf.\
\cite[Exp.~I]{berthelot-grothendieck-illusie-SGA-6},
\cite[Ch.~2]{thomason-trobaugh-higher-K-theory},
\cite[\sptag{064N}]{stacks-project}).
Note that any bounded above complex of finitely generated projective
$\mathcal{A}$-modules is pseudo-coherent.
Let $\D(\mathcal{A})_\pseudocoh$ be the full subcategory of
$\D(\mathcal{A})$ of pseudo-coherent objects. It is a strictly
full triangulated subcategory and Karoubian
(see \cite[\sptag{064V}, \sptag{064X}]{stacks-project}).

We write $\D(\mathcal{A})_\pf$ for the full subcategory of
$\D(\mathcal{A})$ of perfect complexes, i.\,e.\ of objects that
are isomorphic to bounded complexes of finitely generated
projective $\mathcal{A}$-modules.
It is a
strictly
full Karoubian triangulated subcategory of
$\D(\mathcal{A})$.

We write $\D^-(\mathcal{A})$ for the full subcategory of
$\D(\mathcal{A})$ of objects $M$ whose total cohomology
$\bigoplus_{n \in \bZ} \HH^n(M)$ is bounded above.
Similarly, we define $\D^+(\mathcal{A})$ and
$\D^\bd(\mathcal{A})$.
These categories are strictly full Karoubian triangulated
subcategories of
$\D(\mathcal{A})$. We have $\D(\mathcal{A})_\pseudocoh \subset
\D^-(\mathcal{A})$ and $\D(\mathcal{A})_\pf \subset
\D^\bd(\mathcal{A})_\pseudocoh :=
\D(\mathcal{A})_\pseudocoh \cap \D^\bd(\mathcal{A})$.

\begin{remark}
  \label{r:algebra-as-lincat}
  Any algebra $A$ over a field $\kk$ may be viewed as a
  $\kk$-linear category with one object, so all the notions
  just introduced may
  be used for $A$. For example, $\D(A)$ is the derived category of
  the abelian category $\Mod(A)$ of $A$-modules.

  If we assume that $A$ is a noetherian $\kk$-algebra, then
  there is a canonical functor $\D(\mod(A)) \ra
  \D(\Mod(A))$. This functor is obviously
  fully faithful on $\D^\bd(\mod(A))$, and the essential image of
  this category under this functor
  is the full subcategory
  $\D^\bd_{\mod(A)}(A)$
  of $\D(A)$ of objects with bounded
  finitely generated cohomology. Hence we get an equivalence
  \begin{equation}
    \label{eq:DbmodA-DbApscoh}
    \D^\bd(\mod(A)) \sira \D^\bd_{\mod(A)}(A)
  \end{equation}
  of $\kk$-linear triangulated categories; note also that
  \begin{equation}
    \label{eq:Db_modA-DbApscoh}
    \D^\bd_{\mod(A)}(A)=\D^\bd(A)_\pseudocoh.
  \end{equation}
\end{remark}

\subsection{DG categories and smoothness}
\label{sec:dg-categ-smoothn}

Given a dg category $\mathcal{E}$, we denote its homotopy
category by $[\mathcal{E}]$. The derived category of dg
$\mathcal{E}$-modules is denoted by
$\D(\mathcal{E})$. To avoid misunderstandings, we emphasize that
the objects of $\D(\mathcal{E})$ are dg
$\mathcal{E}$-modules.
The full subcategory of $\D(\mathcal{E})$ of compact objects
coincides with $\thick(\mathcal{E})$ and is denoted
$\per(\mathcal{E})$. We remind the reader of the following
definition.

\begin{definition}
  \label{d:smooth}
  A dg category $\mathcal{E}$ is \define{smooth over $\kk$} if
  $\mathcal{E}
  =\leftidx{_\mathcal{E}}{\mathcal{E}}{_\mathcal{E}}\in
  \per(\mathcal{E} \otimes_\kk \mathcal{E}^\opp)$.
\end{definition}

Since a dg algebra is a dg category with precisely one object we
can also speak about $\kk$-smoothness of dg algebras.

\subsection{Smoothness for triangulated categories of modules}
\label{sec:smoothn-triang-categ}

We go back to the setting in section~\ref{sec:modul-over-texorpdfs}
and assume that $\mathcal{A}$ is a $\kk$-linear category.
Let $\ul\hProj(\mathcal{A})$ resp.\ $\ul\hInj(\mathcal{A})$ be the
full dg subcategory of
$\ul\C(\mathcal{A})$
of h-projective resp.\ h-injective objects.
They are dg enhancements of $\D(\mathcal{A})$.

If $\mathcal{T} \subset \D(\mathcal{A})$ is a strictly full
triangulated subcategory we write
$\ul\hProj_\mathcal{T}(\mathcal{A})$ for the full dg subcategory of
$\ul\hProj(\mathcal{A})$ whose objects are in $\mathcal{T}$. This is
a dg enhancement of $\mathcal{T}$. Similarly, we define the
quasi-equivalent dg enhancement
$\ul\hInj_\mathcal{T}(\mathcal{A})$ of $\mathcal{T}$.

\begin{definition}
  \label{d:T-smooth}
  A strictly full triangulated subcategory $\mathcal{T} \subset
  \D(\mathcal{A})$ is \define{smooth over $\kk$} if
  $\ul\hProj_\mathcal{T}(\mathcal{A})$ is a smooth dg
  $\kk$-category.
\end{definition}

\begin{remark}
  \label{r:smoothness-dg-endos-classical-generator}
  Let $\mathcal{T}$ be a strictly full triangulated subcategory of
  $\D(\mathcal{A})$.
  Assume that $P$ is an h-projective
  classical generator of
  $\mathcal{T}$. Then
  $\mathcal{T}$ is smooth over $\kk$ if
  and only if
  the endomorphism dg algebra
  $\ul\C_\mathcal{A}(P,P)$ is smooth over $\kk$
  (see \cite[Prop.~2.18]{valery-olaf-matrix-factorizations-and-motivic-measures}).
\end{remark}

\begin{remark}
  \label{r:DbmodA-smooth}
  If $A$ is a noetherian algebra over a field $\kk$ we also want
  to speak
  about smoothness of $\D^\bd(\mod(A))$. We say that
  $\D^\bd(\mod(A))$ is \define{smooth over $\kk$} if the
  equivalent category
  $\D^\bd_{\mod(A)}(A)$ is smooth over $\kk$ in the sense of the
  above definition (cf.\ equivalence~\eqref{eq:DbmodA-DbApscoh}).
  An equivalent condition is that the standard \textit{projective} dg
  enhancement of
  $\D^\bd(\mod(A))$ by bounded above complexes of finitely
  generated projective $A$-modules with bounded cohomology is a
  smooth dg $\kk$-category.
  Equivalently, we could use the standard \textit{injective}
  dg enhancement by
  bounded below complexes of injective $A$-modules with bounded
  finitely generated cohomology modules.
\end{remark}

\begin{remark}
  If a strictly full triangulated subcategory $\mathcal{T} \subset
  \D(\mathcal{A})$ is smooth over $\kk$ then
  $\mathcal{T}$ has a strong generator. This follows from
  \cite[Lemma~3.5, Lemma~3.6.(a)]{lunts-categorical-resolution}
  and the fact that any smooth dg $\kk$-category has a
  classical generator; we do not prove the last statement here.
  For the categories
  $\D^\bd(\mod(A))$ appearing in
  Theorems~\ref{t:Dbmod-findimalg-separable-smooth} and
  \ref{t:DbmodA-smooth} the existence of a classical
  generator is established in order to prove smoothness (see
  Lemma~\ref{l:class-gen-Dbmod-fin-dim-alg}
  and Theorem~\ref{t:generator-DbcohA}).
\end{remark}

\subsection{Dualizing objects}
\label{sec:dualizing-objects}

Let $\mathcal{A}$ and $\mathcal{B}$ be $\kk$-linear categories.
Consider the dg functor
\begin{align*}
  \otimes
  \colon \ul\C(\mathcal{A}) \otimes \ul\C(\mathcal{B}))
  & \ra \ul\C(\mathcal{A} \otimes \mathcal{B}),\\
  (M,N) & \mapsto M \otimes N,
\end{align*}
defined by $(M\otimes N)(A,B)=M(A) \otimes N(B)$ on objects
$(A,B) \in \mathcal{A} \otimes \mathcal{B}$.
Define the dg functor
\begin{align*}
  \Hom_\mathcal{B}(-,-) \colon
  \ul\C(\mathcal{B}) \otimes \ul\C(\mathcal{A} \otimes
  \mathcal{B})
  & \ra \ul\C(\mathcal{A}),\\
  (N,X)
  & \mapsto \Hom_\mathcal{B}(N,X),
\end{align*}
in the obvious way such that
\begin{equation}
  \label{eq:otimes-HomB}
  \ul\C_{\mathcal{A} \otimes \mathcal{B}}(M \otimes N, X)
  \cong \ul\C_\mathcal{A}(M, \Hom_\mathcal{B}(N,X))
\end{equation}
natural in $M$, $N$ and $X$ as above.

If we identify $\ul\C(\mathcal{A} \otimes \mathcal{B}) \cong
\ul\C(\mathcal{B} \otimes \mathcal{A})$ in the obvious way we
obtain isomorphisms
\begin{align*}
  \ul\C_\mathcal{A}(M, \Hom_\mathcal{B}(N,X))
  & \cong
    \ul\C_{\mathcal{A} \otimes \mathcal{B}}(M \otimes N, X)
  \\
  & \cong
    \ul\C_{\mathcal{B} \otimes \mathcal{A}}(N \otimes M, X)
  \\
  & \cong
  \ul\C_\mathcal{B}(N, \Hom_\mathcal{A}(M,X))
  \\
  & =
  \ul\C_\mathcal{B}^\opp(\Hom_\mathcal{A}(M,X),N).
\end{align*}
Hence we obtain an adjunction
\begin{equation*}
  \Hom_\mathcal{A}(-,X) \colon
  \ul\C(\mathcal{A}) \rightleftarrows
  \ul\C(\mathcal{B})^\opp \colon
  \Hom_\mathcal{B}(-,X)
\end{equation*}
of dg functors
for each $X \in \ul\C(\mathcal{A} \otimes \mathcal{B})$. The unit
and counit of this adjunction are the obvious maps ``into the
bidual with respect to $X$'': the unit $M \ra
\Hom_\mathcal{B}(\Hom_\mathcal{A}(M,X),X)$ is the map sending $m
\in M$
to the evaluation map sending $\mu \in \Hom_\mathcal{A}(M,X)$ to $\mu(m)$,
and the counit has essentially the same description.

On the level of derived categories we obtain an adjunction
\begin{equation}
  \label{eq:Hom-to-X-adjunction}
  \Rd\Hom_\mathcal{A}(-,X) \colon
  \D(\mathcal{A}) \rightleftarrows
  \D(\mathcal{B})^\opp \colon
  \Rd\Hom_\mathcal{B}(-,X).
\end{equation}
Its unit is the obvious natural transformation
\begin{equation*}
  \eta =\eta^X \colon \id \ra
  \Rd\Hom_\mathcal{B}(\Rd\Hom_\mathcal{A}(-,X),X)
\end{equation*}
between endofunctors of $\D(\mathcal{A})$
and its counit is the obvious natural transformation
\begin{equation*}
  \epsilon = \epsilon^X \colon \id \ra
  \Rd\Hom_\mathcal{A}(\Rd\Hom_\mathcal{B}(-,X),X)
\end{equation*}
between endofunctors of $\D(\mathcal{B})$ (strictly speaking, the
counit is the transformation of endofunctors of
$\D(\mathcal{B})^\opp$ obtained by reversing the arrow).

\begin{lemma}
  \label{l:max-cats-adj-equiv}
  Let $\mathcal{A}$, $\mathcal{B}$, $X$,
  $\eta^X$, $\epsilon^X$ be as above.
  Consider the following two full subcategories of
  $\D(\mathcal{A})$ and 
  $\D(\mathcal{B})$ 
  defined by
  \begin{align}
    \label{eq:DA^X}
    \D(\mathcal{A})^X & := \{M \in \D(\mathcal{A}) \mid
    \text{$\eta^X_M$ is an isomorphism}\},\\
    \notag
    \D(\mathcal{B})^X & := \{N \in \D(\mathcal{B}) \mid
    \text{$\epsilon^X_N$ is an isomorphism}\}.
  \end{align}
  Then these two subcategories are thick,
  the adjunction \eqref{eq:Hom-to-X-adjunction} restricts to
  an adjoint equivalence
  \begin{equation*}
    \Rd\Hom_\mathcal{A}(-, X) \colon
    \D(\mathcal{A})^X \rightleftarrows
    (\D(\mathcal{B})^X)^\opp \colon
    \Rd\Hom_\mathcal{B}(-, X),
  \end{equation*}
  and they form
  the biggest pair of subcategories on which
  \eqref{eq:Hom-to-X-adjunction} restricts to
  an adjoint equivalence.
\end{lemma}

\begin{proof}
  The subcategories are clearly thick. The other claims are a
  special case of the following categorical
  Lemma~\ref{l:restrict-adj-to-adj-equiv}.
\end{proof}

\begin{lemma}
  \label{l:restrict-adj-to-adj-equiv}
  Let $(L, R, \eta, \epsilon) \colon \mathcal{C} \ra \mathcal{D}$
  be an adjunction of categories, given by functors
  $L \colon \mathcal{C} \rightleftarrows \mathcal{D} \colon R$
  and unit $\eta \colon \id \ra RL$ and counit
  $\epsilon \colon LR \ra \id$.  Let $\mathcal{C}'$ be the full
  subcategory of $\mathcal{C}$ of objects $C$ such that
  $\eta_C \colon C \ra RLC$ is an isomorphism.  Let
  $\mathcal{D}'$ be the full subcategory of $\mathcal{D}$ of
  objects $D$ such that $\epsilon_D \colon LRD \ra D$ is an
  isomorphism. Then our adjunction restricts to an adjoint
  equivalence
  $L \colon \mathcal{C}' \rightleftarrows \mathcal{D}' \colon R$.
  Moreover, it $\mathcal{C}''$ and $\mathcal{D}''$ are full
  subcategories of $\mathcal{C}$ and $\mathcal{D}$ such that our
  adjunction restricts to an adjoint equivalence
  $L \colon \mathcal{C}'' \rightleftarrows \mathcal{D}'' \colon
  R$, then $\mathcal{C}'' \subset \mathcal{C}'$ and
  $\mathcal{D}'' \subset \mathcal{D}'$.
\end{lemma}

\begin{proof}
  For arbitrary objects $C \in \mathcal{C}$ and $D \in
  \mathcal{D}$ there are commutative diagrams
  \begin{equation*}
    \xymatrix{
      {LC}
      \ar[r]^-{L\eta_C}
      \ar[rd]_-{\id_{LC}}
      &
      {LRLC} \ar[d]^-{\epsilon_{LC}}
      \\
      &
      {LC,}
    }
    \qquad
    \xymatrix{
      {RD}
      \ar[r]^-{\eta_{RD}}
      \ar[rd]_-{\id_{RD}}
      &
      {RLRD} \ar[d]^-{R\epsilon_D}
      \\
      &
      {RD.}
    }
  \end{equation*}
  For $C \in \mathcal{C}'$ the first diagram implies $LC \in
  \mathcal{D}'$. For $D \in \mathcal{D}'$ the second diagram
  implies $RD \in \mathcal{C}'$. This implies that $L$ and $R$
  restrict to the subcategories $\mathcal{C}'$ and
  $\mathcal{D}'$, and these restrictions clearly form an
  adjoint equivalence.
  The last claim is obvious since $\eta$ must be an isomorphism
  on all objects of $\mathcal{C}''$ and $\epsilon$ must be an
  isomorphism on all objects of $\mathcal{D}''$.
\end{proof}

In the following Definition~\ref{d:dualizing}
 we use the above
construction in the special case
that
$\mathcal{B}=\mathcal{A}^\opp$.

\begin{definition}
  \label{d:dualizing}
  Let $\mathcal{A}$ be a $\kk$-linear category and
  $\mathcal{T} \subset \D(\mathcal{A})$ a strictly full
  triangulated
  subcategory.
  A \define{dualizing object}
  (or \define{dualizing bimodule} or
  \define{dualizing complex of bimodules}) \define{for}
  $\mathcal{T}$ is a complex $\mathscr{D}$ of
  $\mathcal{A} \otimes
  \mathcal{A}^\opp$-modules such that
  every object of $\mathcal{T}$ is contained in
  the category $\D(\mathcal{A})^\mathscr{D}$ defined in
  \eqref{eq:DA^X}:
  This just means that
  the unit
  \begin{equation*}
    T \sira
    \Rd\Hom_{\mathcal{A}^\opp}(\Rd\Hom_\mathcal{A}(T,\mathscr{D}),
    \mathscr{D})
  \end{equation*}
  is an isomorphism in $\D(\mathcal{A})$
  for all objects $T \in \mathcal{T}$.
  Given a dualizing object $\mathscr{D}$ for $\mathcal{T}$ we
  denote the essential image of $\mathcal{T}$ under the functor
  \begin{equation*}
    \Rd\Hom_\mathcal{A}(-,\mathscr{D}) \colon
    \D(\mathcal{A}) \rightarrow
    \D(\mathcal{A}^\opp)^\opp
  \end{equation*}
  by $\mathcal{T}^\vee=\mathcal{T}^{\vee, \mathscr{D}}$ and call
  it the \define{dual of $\mathcal{T}$ (with respect to
  $\mathscr{D}$)}. Here we
  view $\mathcal{T}^\vee$ as a full subcategory of
  $\D(\mathcal{A}^\opp)$ (and not of its opposite category).
\end{definition}

\begin{remark}
  \label{r:dualizing-on-classical-generator}
  If, in the setting of Definition~\ref{d:dualizing},
  $\mathcal{T}$ is classically generated by an object $E$,
  then $\mathscr{D}$ is a dualizing object for $\mathcal{T}$ if
  and only if
  \begin{equation*}
    E \ra
    \Rd\Hom_{\mathcal{A}^\opp}(\Rd\Hom_\mathcal{A}(E,\mathscr{D}),
    \mathscr{D})
  \end{equation*}
  is an isomorphism. This follows immediately from
  Lemma~\ref{l:max-cats-adj-equiv} because
  $\D(\mathcal{A})^\mathscr{D}$ is a thick subcategory of
  $\D(\mathcal{A})$.
\end{remark}

\begin{remark}
  \label{r:dualizing-object-duality}
  Let $\mathcal{A}$ be a $\kk$-linear category and
  $\mathscr{D}$ a dualizing object for a strictly full
  triangulated subcategory $\mathcal{T} \subset
  \D(\mathcal{A})$ with dual $\mathcal{T}^\vee$.
  Then, by Lemma~\ref{l:max-cats-adj-equiv},
  the adjunction~\eqref{eq:Hom-to-X-adjunction} (for
  $X=\mathscr{D}$ and $\mathcal{B}=\mathcal{A}^\opp$)
  restricts to an
  adjoint equivalence
  \begin{equation*}
    \Rd\Hom_\mathcal{A}(-,\mathscr{D}) \colon
    \mathcal{T} \rightleftarrows
    (\mathcal{T}^\vee)^\opp \colon
    \Rd\Hom_{\mathcal{A}^\opp}(-,\mathscr{D}).
  \end{equation*}
\end{remark}

\begin{remark}
  \label{r:dualizing-object-from-duality}
  Let $\mathcal{A}$ be a $\kk$-linear category and let
  $\mathcal{T} \subset \D(\mathcal{A})$ and
  $\mathcal{S} \subset \D(\mathcal{A}^\opp)$ be strictly full
  triangulated subcategories.
  Let $\mathscr{D}$ be a complex of
  $\mathcal{A} \otimes
  \mathcal{A}^\opp$-modules
  such that
  the adjunction~\eqref{eq:Hom-to-X-adjunction} (for
  $X=\mathscr{D}$ and $\mathcal{B}=\mathcal{A}^\opp$)
  restricts to an adjoint equivalence
  \begin{equation*}
    \Rd\Hom_\mathcal{A}(-,\mathscr{D}) \colon
    \mathcal{T} \rightleftarrows
    \mathcal{S}^\opp \colon
    \Rd\Hom_{\mathcal{A}^\opp}(-,\mathscr{D}).
  \end{equation*}
  Then $\mathscr{D}$ is a dualizing object for $\mathcal{T}$ and
  $\mathcal{S}$ is the dual of $\mathcal{T}$, i.\,e.\
  $\mathcal{S}=\mathcal{T}^\vee$.
  This follows immediately from Lemma~\ref{l:max-cats-adj-equiv}.
\end{remark}

\begin{lemma}
  \label{l:HomB-to-h-inj-preserves-acyclics}
  Let $\mathcal{A}$ and $\mathcal{B}$ be $\kk$-linear categories.
  Let $X \in \C(\mathcal{A} \otimes \mathcal{B})$ be an
  h-injective complex. Then the functor
  \begin{equation*}
    \Hom_\mathcal{B}(-,X) \colon
    \C(\mathcal{B})
    \ra \C(\mathcal{A})
  \end{equation*}
  lands in the subcategory of h-injective complexes and preserves
  acyclic complexes.
\end{lemma}

\begin{proof}
  Let $N \in \C(\mathcal{B})$ be any object.
  We first show that $\Hom_\mathcal{B}(N,X)$
  is h-injective. Given any acyclic object $M \in
  \C(\mathcal{A})$
  we need to see that
  \begin{equation*}
    \ul\C_\mathcal{A}(M, \Hom_\mathcal{B}(N,X))
    \overset{\eqref{eq:otimes-HomB}}{\cong}
    \ul\C_{\mathcal{A} \otimes \mathcal{B}}(M \otimes N, X)
  \end{equation*}
  is acyclic. Since $X$ is h-injective it suffices to show that
  $M \otimes N$ is acyclic.
  But this is true since $\kk$ is a field.

  Now assume that $N$ is acyclic.
  We show that
  $\Hom_\mathcal{B}(N,X)$ is acyclic. Given an
  arbitrary object $A \in
  \mathcal{A}$ we need to see that
  \begin{equation*}
    \Hom_\mathcal{B}(N,X)(A)
    =
    \ul\C_\mathcal{A}(\Yo(A), \Hom_\mathcal{B}(N,X))
    \overset{\eqref{eq:otimes-HomB}}{\cong}
    \ul\C_{\mathcal{A} \otimes \mathcal{B}}(\Yo(A) \otimes N, X)
  \end{equation*}
  is acyclic. Since $X$ is h-injective it suffices to see that
  $\Yo(A) \otimes N$ is acyclic. But this is true since $N$ is
  acyclic and $\kk$ is a field.
\end{proof}

\subsection{Tensor product and dg homomorphisms}
\label{sec:tensor-product-dg}

We prove a key technical result which uses the notion of
pseudo-coherence.

\begin{proposition}
  \label{p:MI-otimes-NJ-qiso-MN-K}
  Let $\mathcal{A}$ and $\mathcal{B}$ be $\kk$-linear categories.
  Let $M \in \C(\mathcal{A})$ and $N \in \C(\mathcal{B})$.
  Let $I \in \C(\mathcal{A})$ and $J \in \C(\mathcal{B})$ be
  h-injective objects
  and
  let $\kappa \colon I \otimes J \ra K$ be a quasi-isomorphism to
  an h-injective object $K \in \C(\mathcal{A} \otimes \mathcal{B})$.
  Assume that one of the following two conditions is satisfied.
  \begin{enumerate}
  \item
    \label{enum:MNperfect}
    $M \in \D(\mathcal{A})_\pf$ and $N \in \D(\mathcal{B})_\pf$;
  \item
    \label{enum:pseudocohMNIJ}
    $M \in \D(\mathcal{A})_\pseudocoh$ and
    $N \in \D(\mathcal{B})_\pseudocoh$ and
    $I \in \D^+(\mathcal{A})$ and
    $J \in \D^+(\mathcal{B})$.
  \end{enumerate}
  Then the composition
  \begin{equation}
    \label{eq:47}
    \ul\C_\mathcal{A}(M,I) \otimes \ul\C_\mathcal{B}(N,J)
    \xra{\otimes}
    \ul\C_{\mathcal{A} \otimes \mathcal{B}}(M \otimes N, I \otimes J)
    \xra{\kappa_*}
    \ul\C_{\mathcal{A} \otimes \mathcal{B}}(M \otimes N, K)
  \end{equation}
  in $\C(\kk)$ is a quasi-isomorphism.
\end{proposition}

\begin{proof}
  Note that $\ul\C_\mathcal{C}(-,L)$
  preserves quasi-isomorphisms if $L \in \C(\mathcal{C})$ is an
  h-injective complex of modules over a $\kk$-linear category
  $\mathcal{C}$, that
  $(I' \otimes -) \colon \C(\mathcal{B}) \ra
  \C(\mathcal{A} \otimes \mathcal{B})$
  and
  $(- \otimes J') \colon \C(\mathcal{A}) \ra
  \C(\mathcal{A} \otimes \mathcal{B})$
  preserve quasi-isomorphisms, for $I' \in \C(\mathcal{A})$ and
  $J' \in
  \C(\mathcal{B})$, and that tensoring over $\kk$ preserves
  quasi-isomorphisms, since $\kk$ is a field.
  Hence we can replace $M$ and $N$ by isomorphic objects in
  $\D(\mathcal{A})$ and $\D(\mathcal{B})$, respectively.

  If \ref{enum:MNperfect} holds, $M$ and $N$ can be assumed to be
  bounded complexes of finitely generated projective
  modules. Using brutal truncation,
  shifts,
  and the fact that any finitely generated projective module is a
  direct summand of a finitely generated free module we reduce to
  the case that $M=\Yo(A)=\mathcal{A}(-,A)$ for some $A \in
  \mathcal{A}$ and
  $N=\Yo(B)=\mathcal{B}(-,B)$ for some $B \in \mathcal{B}$. But
  in this case
  \eqref{eq:47} is isomorphic to the quasi-isomorphism
  $I(A) \otimes J(B) = I(A) \otimes J(B) \ra K(A,B)$.

  Now assume that \ref{enum:pseudocohMNIJ} holds.
  Observe that $I$ and $J$ are homotopy equivalent to bounded
  below
  complexes of injective modules. Using this it is easy to see
  that we can assume that $I$, $J$ and $K$
  are bounded below complexes of injective modules.
  We can also assume that $M$ and $N$ are bounded above complexes
  of finitely generated free modules since $M$ and $N$ are
  pseudocoherent.

  When checking that the composition in
  \eqref{eq:47} induces an
  isomorphism on the $n$-th cohomology,
  for a fixed $n \in \bZ$, only finitely many components
  of $M$ and $N$ matter. Hence we can assume that $M$ and
  $N$ are
  bounded complexes of finitely generated free modules. Then $M$
  and $N$ are
  perfect
  and we can use \ref{enum:MNperfect}.
\end{proof}

\subsection{A sufficient condition for smoothness}
\label{sec:suff-cond-smoothn}

We now state the main theorem of this article.

\begin{theorem}
  \label{t:sufficient-for-smoothness}
  Let $\mathcal{A}$ be a $\kk$-linear category where $\kk$ is a
  field.
  Then a strictly full
  triangulated subcategory
  $\mathcal{T} \subset \D(\mathcal{A})$
  is smooth over $\kk$ if there is a dualizing object
  $\mathscr{D} \in \D(\mathcal{A} \otimes \mathcal{A}^\opp)$
  for $\mathcal{T}$
  such that the following two conditions are satisfied:
  \begin{enumerate}[label=(A\arabic*)]
  \item
    \label{enum:T-class-gen}
    $\mathcal{T}$ is classically generated by an object of
    $\D^\bd(\mathcal{A})_\pseudocoh$ whose
    image under
    the equivalence
    \begin{equation*}
      \Rd\Hom_\mathcal{A}(-,\mathscr{D}) \colon
      \mathcal{T} \sira
      (\mathcal{T}^\vee)^\opp
    \end{equation*}
    from
    Remark~\ref{r:dualizing-object-duality}
    is in
    $\D^\bd(\mathcal{A}^\opp)_\pseudocoh$;
  \item
    \label{enum:dualizing-object-from-T-Tvee}
    $\mathscr{D}$ is contained in the thick subcategory
    of $\D(\mathcal{A} \otimes \mathcal{A}^\opp)$
    generated
    by the essential image of the
    functor
    \begin{equation*}
      \otimes \colon \mathcal{T} \times
      \mathcal{T}^\vee
      \ra
      \D(\mathcal{A} \otimes \mathcal{A}^\opp).
    \end{equation*}
  \end{enumerate}
\end{theorem}

\begin{remark}
  \label{r:on-t:sufficient-for-smoothness}
  Condition~\ref{enum:T-class-gen} clearly implies $\mathcal{T}
  \subset
  \D^\bd(\mathcal{A})_\pseudocoh$ and $\mathcal{T}^\vee \subset
  \D^\bd(\mathcal{A}^\opp)_\pseudocoh$ because the image of a
  classical generator of $\mathcal{T}$ under the equivalence is a
  classical generator of
  $\mathcal{T}^\vee$.

  If $E$ is a classical generator of $\mathcal{T}$ and
  $F$ is a classical generator of $\mathcal{T}^\vee$, then
  condition~\ref{enum:dualizing-object-from-T-Tvee}
  is clearly equivalent to:
  \begin{enumerate}[label=(A\arabic*)', start=2]
  \item
    \label{enum:dualizing-object-from-class-gen-T-Tvee}
    $\mathscr{D}$ is contained in the thick subcategory
    of $\D(\mathcal{A} \otimes \mathcal{A}^\opp)$
    generated by $E \otimes F$.
  \end{enumerate}
\end{remark}

\begin{proof}
  We first reduce to the case that $\mathcal{T}$ is Karoubian.
  Obviously,
  $\mathscr{D}$ is also a dualizing object for
  $\thick(\mathcal{T})=\thick_{\D(\mathcal{A})}(\mathcal{T})$,
  and the 
  assumptions~\ref{enum:T-class-gen}
  and \ref{enum:dualizing-object-from-T-Tvee}
  are also satisfied by the pair
  $(\thick(\mathcal{T}), \mathscr{D})$.
  Moreover, a classical generator of $\mathcal{T}$ is
  certainly a classical generator of
  $\thick(\mathcal{T})$, so
  $\mathcal{T}$ is smooth over $\kk$ if and only if
  $\thick(\mathcal{T})$ is smooth over $\kk$,
  by Remark~\ref{r:smoothness-dg-endos-classical-generator}.

  Hence, by
  replacing $\mathcal{T}$ by $\thick(\mathcal{T})$
  we can and
  will assume
  in the following that $\mathcal{T}$ is Karoubian.
  Without loss of generality we assume that
  $\mathscr{D}$ is an h-injective complex of $\mathcal{A} \otimes
  \mathcal{A}^\opp$-modules.

  Consider the adjunction
  \begin{equation*}
    (-)^\vee:=\Hom_{\mathcal{A}}(-,\mathscr{D})
    \colon
    \ul\C(\mathcal{A}) \rightleftarrows
    \ul\C(\mathcal{A}^\opp)^\opp
    \colon (-)^\vee:=\Hom_{\mathcal{A}^\opp}(-,\mathscr{D})
  \end{equation*}
  of dg functors.
  Both functors preserve quasi-isomorphisms,
  by Lemma~\ref{l:HomB-to-h-inj-preserves-acyclics},
  and therefore our adjunction descends straightforwardly
  to an adjunction
  \begin{equation*}
    (-)^\vee=\Hom_{\mathcal{A}}(-,\mathscr{D})
    \colon
    \D(\mathcal{A}) \rightleftarrows
    \D(\mathcal{A}^\opp)^\opp
    \colon
    (-)^\vee=\Hom_{\mathcal{A}^\opp}(-,\mathscr{D})
  \end{equation*}
  of triangulated functors. This is (up to unique isomorphism) the
  adjunction~\eqref{eq:Hom-to-X-adjunction} (for
  $\mathcal{B}=\mathcal{A}^\opp$ and $X=\mathscr{D}$ there).

  Since $\mathscr{D}$ is a dualizing object for $\mathcal{T}$,
  this adjunction restricts
  by Remark~\ref{r:dualizing-object-duality}
  to an adjoint equivalence
  \begin{equation}
    \label{eq:T-Tvee-equi}
    (-)^\vee=\Hom_{\mathcal{A}}(-,\mathscr{D})
    \colon
    \mathcal{T}
    \rightleftarrows
    (\mathcal{T}^\vee)^\opp
    \colon
    (-)^\vee=\Hom_{\mathcal{A}^\opp}(-,\mathscr{D}).
  \end{equation}

  \textbf{Claim:}
  For any h-projective object $Q \in \mathcal{T}$
  and any $R \in \mathcal{T}$ the morphism
  \begin{equation*}
    \ul\C_\mathcal{A}(Q,R)
    \xra{(-)^\vee}
    \ul\C_{\mathcal{A}^\opp}^\opp(Q^\vee,R^\vee)=
    \ul\C_{\mathcal{A}^\opp}(R^\vee,Q^\vee)
  \end{equation*}
  is a quasi-isomorphism.

  Indeed, the induced map
  on the $n$-th
  cohomology is given by
  \begin{equation*}
    [\ul\C_\mathcal{A}](Q,[n]R)
    \xra{(-)^\vee}
    [\ul\C_{\mathcal{A}^\opp}](([n]R)^\vee,Q^\vee).
  \end{equation*}
  Since $Q$ is h-projective and
  $Q^\vee=\Hom_\mathcal{A}(Q,\mathscr{D})$ is h-injective,
  by Lemma~\ref{l:HomB-to-h-inj-preserves-acyclics},
  this map
  is identified with the map
  \begin{equation*}
    \D_\mathcal{A}(Q,[n]R)
    \xra{(-)^\vee}
    \D_{\mathcal{A}^\opp}(([n]R)^\vee,Q^\vee)
  \end{equation*}
  which is an isomorphism by the
  equivalence~\eqref{eq:T-Tvee-equi}.
  This proves the claim.

  Let $P \in \D^\bd(\mathcal{A})_\pseudocoh$
  be a classical generator
  of $\mathcal{T}$ such that $P^\vee \in
  \D^\bd(\mathcal{A}^\opp)_\pseudocoh$; such an
  object exists by assumption~\ref{enum:T-class-gen}.
  Additionally, we can assume that $P$ is h-projective.
  By Remark~\ref{r:smoothness-dg-endos-classical-generator} we
  need to show $\kk$-smoothness of the dg algebra
  \begin{equation*}
    \mathcal{E}:=\ul\C_\mathcal{A}(P,P).
  \end{equation*}

  The above claim shows that
  \begin{equation}
    \label{eq:E-Evee-qiso}
    \mathcal{E}=\ul\C_\mathcal{A}(P,P)
    \xra{(-)^\vee}
    \ul\C_{\mathcal{A}^\opp}^\opp(P^\vee,P^\vee)
  \end{equation}
  is a quasi-isomorphism of dg algebras.

  Note that the obvious map
  \begin{equation*}
    P \ra (P^\vee)^\vee
  \end{equation*}
  is a quasi-isomorphism in $\C(\mathcal{A})$ because it becomes
  an isomorphism in $\D(\mathcal{A})$ by the assumption that
  $\mathscr{D}$ is a dualizing object.
  If we apply $\ul\C(P,-)$ to this quasi-isomorphism
  we obtain a quasi-isomorphism
  \begin{equation}
    \label{eq:E-(P,Pveevee)}
    \mathcal{E}=\ul\C_\mathcal{A}(P,P)
    \ra
    \ul\C_\mathcal{A}(P, (P^\vee)^\vee)
  \end{equation}
  of dg $\mathcal{E}$-modules because $P$ is h-projective.
  Here $\mathcal{E}$ acts from the right.
  There is also a natural left action of the dg algebra
  $\mathcal{E}$ on
  $\mathcal{E}$ and on
  $\ul\C_\mathcal{A}(P, (P^\vee)^\vee)$, the action on the latter
  coming from
  the morphism
  \begin{equation*}
    \mathcal{E}=\ul\C(P,P) \xra{((-)^\vee)^\vee}
    \ul\C((P^\vee)^\vee,(P^\vee)^\vee)
  \end{equation*}
  of dg algebras. It is easy to check that
  \eqref{eq:E-(P,Pveevee)} is compatible with these actions and
  hence a quasi-isomorphism of dg $\mathcal{E} \otimes
  \mathcal{E}^\opp$-modules.

  Let $P \ra I$ be a quasi-isomorphism to an h-injective object
  $I \in \C(\mathcal{A})$.
  Let $\kappa \colon I \otimes P^\vee \ra K$ be
  a quasi-isomorphism to an h-injective object
  $K \in \C(\mathcal{A} \otimes \mathcal{A}^\opp)$.
  Then
  \begin{equation*}
    \lambda \colon P \otimes P^\vee \ra I \otimes P^\vee \xra{\kappa} K
  \end{equation*}
  is a quasi-isomorphism because $\kk$ is a field.
  Now consider the following commutative diagram with obvious
  maps
  where $\beta$ is defined so that the triangle containing
  $\beta$ is commutative.
  \begin{equation*}
    \xymatrix{
      {\mathcal{E} \otimes \mathcal{E}^\opp}
      \ar[d]_-{\id \otimes (-)^\vee}
      \ar[rd]^-\beta
      \\
      {\ul\C_\mathcal{A}(P,P) \otimes
        \ul\C_{\mathcal{A}^\opp}(P^\vee, P^\vee)}
      \ar[r]^-{\otimes}
      \ar[d]
      &
      {\ul\C_{\mathcal{A} \otimes \mathcal{A}^\opp}(P \otimes
        P^\vee, P \otimes P^\vee)}
      \ar[r]^-{\lambda_*}
      &
      {\ul\C_{\mathcal{A} \otimes \mathcal{A}^\opp}(P \otimes
        P^\vee, K)}
      \ar@{=}[d]
      \\
      {\ul\C_\mathcal{A}(P,I) \otimes
        \ul\C_{\mathcal{A}^\opp}(P^\vee, P^\vee)}
      \ar[r]^-{\otimes}
      &
      {\ul\C_{\mathcal{A} \otimes \mathcal{A}^\opp}(P \otimes
        P^\vee, I \otimes P^\vee)}
      \ar[r]^-{\kappa_*}
      &
      {\ul\C_{\mathcal{A} \otimes \mathcal{A}^\opp}(P \otimes
        P^\vee, K)}
    }
  \end{equation*}
  Its vertical arrows are quasi-isomorphisms: this uses the
  quasi-isomorphism
  \eqref{eq:E-Evee-qiso}, the fact that $P$ is h-projective,
  and the fact
  that $\kk$ is a field.
  The composition of the lower row is a quasi-isomorphism by
  Proposition~\ref{p:MI-otimes-NJ-qiso-MN-K}.\ref{enum:pseudocohMNIJ};
  we use here that $I$ and $P^\vee$ are h-injective
  (by Lemma~\ref{l:HomB-to-h-inj-preserves-acyclics}) and that
  $I \sila P \in \D^\bd(\mathcal{A})_\pseudocoh$
  and $P^\vee \in
  \D^\bd(\mathcal{A})_\pseudocoh$
  by assumption~\ref{enum:T-class-gen}. Hence the composition
  $\lambda_* \circ \beta$ is
  a quasi-isomorphism.

  We now apply
  \cite[Lemma~B.1.(a)]{valery-olaf-new-enhancements}
  to the dg category $\ul\C(\mathcal{A} \otimes
  \mathcal{A}^\opp)$, its full pretriangulated dg subcategory
  $\ul\hInj_{\mathcal{S}}(\mathcal{A} \otimes
  \mathcal{A}^\opp)$ where $\mathcal{S}=\thick(P \otimes
  P^\vee)$, the quasi-isomorphism $\lambda \colon P \otimes
  P^\vee \ra K$ in
  $\C(\mathcal{A}
  \otimes \mathcal{A}^\opp)$, and the morphism $\beta$ of dg
  algebras from the
  above diagram. This yields the equivalence
  \begin{equation*}
    \res_{\mathcal{E} \otimes
      \mathcal{E}^\opp}^{\ul\C_{\mathcal{A} \otimes
        \mathcal{A}^\opp}(P \otimes P^\vee, P \otimes P^\vee)}
    \circ \ul\C_{\mathcal{A} \otimes
      \mathcal{A}^\opp}(P \otimes P^\vee, -) \colon
    [\ul\hInj_{\mathcal{S}}(\mathcal{A} \otimes
    \mathcal{A}^\opp)]
    \sira
    \per(\mathcal{E} \otimes \mathcal{E}^\opp)
  \end{equation*}
  of triangulated categories.

  By assumption~\ref{enum:dualizing-object-from-T-Tvee},
  the object $\mathscr{D}$ is in $\mathcal{S}$. We have
  assumed that $\mathscr{D}$ is h-injective, so
  $\mathscr{D} \in \ul\hInj_{\mathcal{S}}(\mathcal{A} \otimes
  \mathcal{A}^\opp)$.
  Hence
  \begin{equation*}
    \ul\C_{\mathcal{A} \otimes
      \mathcal{A}^\opp}(P \otimes P^\vee, \mathscr{D})
    \in
    \per(\mathcal{E} \otimes \mathcal{E}^\opp)
  \end{equation*}
  by the above equivalence.
  The isomorphisms
  \begin{align*}
    \ul\C_{\mathcal{A} \otimes
    \mathcal{A}^\opp}(P \otimes P^\vee, \mathscr{D})
    & =
      \ul\C_{\mathcal{A} \otimes
      \mathcal{A}^\opp}(P \otimes \Hom_{\mathcal{A}}(P,
      \mathscr{D}), \mathscr{D})
    \\
    & \cong
      \ul\C_\mathcal{A}(P, \Hom_{\mathcal{A}^\opp}(\Hom_{\mathcal{A}}(P,
      \mathscr{D}), \mathscr{D}))
    & \text{(by \eqref{eq:otimes-HomB})}
    \\
    & =
      \ul\C_\mathcal{A}(P, (P^\vee)^\vee)
  \end{align*}
  of dg $\mathcal{E} \otimes \mathcal{E}^\opp$-modules and the
  quasi-isomorphism
  \begin{equation*}
    \mathcal{E}=\ul\C_\mathcal{A}(P,P)
    \xra{\eqref{eq:E-(P,Pveevee)}}
    \ul\C_\mathcal{A}(P, (P^\vee)^\vee)
  \end{equation*}
  of dg $\mathcal{E} \otimes \mathcal{E}^\opp$-modules then show
  $\mathcal{E} \in \per(\mathcal{E} \otimes \mathcal{E}^\opp)$,
  i.\,e.\ $\mathcal{E}$ is smooth over $\kk$.
\end{proof}

\section{Smoothness for finite dimensional algebras}
\label{sec:smoothn-finite-dimen}

We denote the Jacobson radical of a ring $R$ by $\rad(R)$.
Recall that a ring is semisimple if and only if it is Artinian and its
Jacobson radical is zero (see e.\,g.\
\cite[Thm.~2.2]{farb-dennis-NC-algebra}). In particular, for a
finite-dimensional algebra $A$ over a field $\frac{A}{\rad(A)}$ is
semisimple.

Given a finite-dimensional $\kk$-algebra $A$ remember that
$\mod(A)$ is the abelian category of finite-dimensional
$A$-modules and that $\D^\bd(\mod(A))$ is the full subcategory of
its derived category $\D(\mod(A))$ of objects with bounded
cohomology.

\begin{lemma}
  \label{l:class-gen-Dbmod-fin-dim-alg}
  Let $A$ be a finite-dimensional algebra over a field $\kk$.
  Then $\D^\bd(\mod(A))$ has a classical generator, for example
  the direct sum of representatives of the set of simple
  $A$-modules up to isomorphism, or $\frac{A}{\rad(A)}$.
\end{lemma}

\begin{proof}
  Any object of $\D^\bd(\mod(A))$ is built from its finitely many
  non-zero cohomology modules, and each such module has a
  finite filtration with simple subquotients (a composition
  series).
  Since each simple $A$-module appears in such a
  composition series of $A$, there are, up to isomorphism, only
  finitely many
  simple $A$-modules, say $S_1, \dots, S_r$. Then
  $\bigoplus_{i=1}^r S_i$ is a classical generator of
  $\D^\bd(\mod(A))$.
  Since $\frac{A}{\rad(A)}$ is a semisimple $A$-module which contains
  each $S_i$ with positive multiplicity, $\frac{A}{\rad(A)}$ is a
  classical generator as well.
\end{proof}

Recall the Artin--Wedderburn theorem saying that every semisimple
ring $A$ is isomorphic to a finite
product $\prod_{i=1}^r \M_{n_i}(D_i)$
of matrix rings
over division rings $D_i$, for
suitable $n_i \in \bN_{>0}$. In particular, the
center $\Z(A)$ of $A$ is then isomorphic to the product $\prod_{i=1}^r
Z(D_i)$ of fields.
If $A$ is an algebra over the field $\kk$, we get
field extensions $\kk \hra Z(D_i)$; these field extensions are unique up to
isomorphism and order.

\begin{definition}
  [{\cite[Def.\ on page 89]{farb-dennis-NC-algebra}}]
  \label{d:sep-alg-over-field}
  Let $\kk$ be a field. A $\kk$-algebra $A$ is
  \define{separable (over $\kk$)}
  if and only if
  $A$ is a finite-dimensional semisimple $\kk$-algebra such
  that
  each field extension $\kk
  \hra Z(D_i)$ is separable
  if we fix an isomorphism $A \cong \prod_{i=1}^r \M_{n_i}(D_i)$
  as above.
\end{definition}

\begin{remark}
  In particular, if $\kk$ is perfect, than a
  $\kk$-algebra is separable if and only if it is
  finite-dimensional and semisimple.
\end{remark}

\begin{remark}
  \label{r:sep-alg-over-field}
  There is a general definition of a \textit{separable algebra over a
  commutative ring}, see
  \cite{auslander-goldman-brauer} or
  \cite[Ch.~III]{knus-ojanguren-descente-azumaya}.
  For algebras over a field this general definition is
  equivalent to
  Definition~\ref{d:sep-alg-over-field}, by
  \cite[Thm.~III.3.1]{knus-ojanguren-descente-azumaya}.
  Note also that
  Definition~\ref{d:sep-alg-over-field} generalizes the
  usual definition of a separable field extension.
\end{remark}

\begin{proposition}
  \label{p:otimes-rad-separable-alg}
  Let $A$ and $B$ be finite-dimensional $\kk$-algebras.
  Assume that at least one of $\frac{A}{\rad(A)}$
  and $\frac{B}{\rad(B)}$ is a separable
  $\kk$-algebra. Then
  $\frac{A}{\rad(A)} \otimes_\kk \frac{B}{\rad(B)}$
  is a semisimple $\kk$-algebra
  and
  \begin{equation}
    \label{eq:rad-otimes}
    \frac{A}{\rad(A)} \otimes_\kk \frac{B}{\rad(B)}
    =
    \frac{A \otimes_\kk B}{\rad (A \otimes_\kk B)}
  \end{equation}
  canonically. The displayed $\kk$-algebra is separable if both
  $\frac{A}{\rad(A)}$ and
  $\frac{B}{\rad(B)}$ are separable over $\kk$.
\end{proposition}

\begin{proof}
  It is well-known that the tensor product of a separable
  $\kk$-algebra with 
  a finite-dimensional semisimple $\kk$-algebra is again
  semisimple, see \cite[Prop.~3.9]{farb-dennis-NC-algebra}.
  This shows that the tensor product $\frac{A}{\rad(A)} \otimes_\kk
  \frac{B}{\rad(B)}$
  is semisimple if one of the factors is separable.
  If both factors are separable over $\kk$, then so is the tensor
  product
  $\frac{A}{\rad(A)} \otimes_\kk
  \frac{B}{\rad(B)}$
  by \cite[Prop.~1.5]{auslander-goldman-brauer} (for
  $R=R_1=R_2=\kk$ there) (using Remark~\ref{r:sep-alg-over-field}).

  We now deduce equality \eqref{eq:rad-otimes} from
  semisimplicity of
  \begin{equation*}
    \frac{A}{\rad(A)} \otimes_\kk \frac{B}{\rad(B)} =
    \frac{A \otimes_\kk B}{\rad(A) \otimes_\kk B + A \otimes_\kk
      \rad(B)}.
  \end{equation*}
  If we can show that
  \begin{equation*}
    I:=\rad(A) \otimes_\kk B + A \otimes_\kk
      \rad(B)
  \end{equation*}
  is a nilpotent ideal, then $I=\rad(A \otimes_\kk B)$
  by \cite[Prop.~I.3.3]{ARS-rep-artin-algebras}, yielding equality
  \eqref{eq:rad-otimes}.
  Observe that
  \begin{equation*}
    I^n
    =
    \sum_{i=0}^n
    \rad(A)^i \otimes_\kk \rad(B))^{n-i}
  \end{equation*}
  for any $n \in \bN$.
  Since $\rad(A)$ and $\rad(B)$ are nilpotent two-sided
  ideals, by \cite[Prop.~I.3.1]{ARS-rep-artin-algebras}, we see
  that $I$ is nilpotent as well. This proves the proposition.
\end{proof}

\begin{corollary}
  \label{c:otimes-rad-separable-alg}
  Let $A$ and $B$ be finite-dimensional $\kk$-algebras such that
  at least one of $\frac{A}{\rad(A)}$ and $\frac{B}{\rad(B)}$ is
  a separable
  $\kk$-algebra.
  If $E$ and $F$ are classical generators of $\D^\bd(\mod(A))$ and
  $\D^\bd(\mod(B))$, respectively,
  then $E \otimes_\kk F$
  is a classical generator of
  $\D^\bd(\mod(A \otimes_\kk B))$.
\end{corollary}

\begin{proof}
  Recall that $\ol{A}:=\frac{A}{\rad(A)}$ and $\ol{B}:=\frac{B}{\rad(B)}$ are classical
  generators of $\D^\bd(\mod(A))$ and
  $\D^\bd(\mod(B))$, respectively, by
  Lemma~\ref{l:class-gen-Dbmod-fin-dim-alg}.
  From $\ol{A} \in \thick(E)$ we obtain
  $\ol{A} \otimes_\kk F \in  \thick(E \otimes_\kk F)$.
  From $\ol{B} \in \thick(F)$ we obtain
  $\ol{A} \otimes_\kk \ol{B} \in  \thick(\ol{A} \otimes_\kk F) \subset
  \thick(E \otimes_\kk F)$.
  Hence it is sufficient to show that
  $\ol{A} \otimes_\kk \ol{B}$ is a classical generator of
  $\D^\bd(\mod(A \otimes_\kk B))$.
  But $\frac{A \otimes_\kk B}{\rad (A \otimes_\kk B)}$
  is a classical generator of this category, by
  Lemma~\ref{l:class-gen-Dbmod-fin-dim-alg}, and
  $\ol{A} \otimes_\kk \ol{B} = \frac{A \otimes_\kk B}{\rad (A
    \otimes_\kk B)}$ by
  Proposition~\ref{p:otimes-rad-separable-alg}.
\end{proof}

\begin{theorem}
  \label{t:Dbmod-findimalg-separable-smooth}
  Let $A$ be a finite-dimensional algebra over a field $\kk$ such
  that $\frac{A}{\rad(A)}$ is separable over $\kk$ (this
  condition is automatic if $\kk$ is perfect).
  Then
  $\D^\bd(\mod(A))$ is smooth over $\kk$ (in the sense defined in
  Remark~\ref{r:DbmodA-smooth}).
\end{theorem}

The idea of proof is as follows. The standard equivalence ``take
the $\kk$-linear dual''
$\D^\bd(\mod(A)) \sira (\D^\bd(\mod(A^\opp)))^\opp$ is
equivalently obtained from the dualizing bimodule
$A= \leftidx{_A}{A}{_A}$. This bimodule is an object of
$\D^\bd(\mod(A \otimes_\kk A^\opp))$, and this category coincides, by
our separability assumption, with its thick subcategory generated
by all tensor products of objects of $\D^\bd(\mod(A))$ and
$\D^\bd(\mod(A^\opp))$. Moreover, $\D^\bd(\mod(A))$ has a classical
generator. This shows that the sufficient condition for
smoothness
of Theorem~\ref{t:sufficient-for-smoothness} is satisfied
for $\D^\bd(\mod(A))$.

\begin{proof}
  Remember that $\D^\bd(\mod(A))$ is equivalent to the category
  \begin{equation}
    \label{eq:T}
    \mathcal{T}:=\D^\bd_{\mod(A)}(A)=\D^\bd(A)_\pseudocoh
  \end{equation}
  (see equivalence \eqref{eq:DbmodA-DbApscoh} and equality
  \eqref{eq:Db_modA-DbApscoh}
  in Remark~\ref{r:algebra-as-lincat}).
  By our definition
  of $\kk$-smoothness of $\D^\bd(\mod(A))$
  in
  Remark~\ref{r:DbmodA-smooth}
  we need to prove that $\mathcal{T}$ is $\kk$-smooth.
  We will use the sufficient condition for smoothness of
  Theorem~\ref{t:sufficient-for-smoothness}.

  If $M$ is a right $A$-module, then its $\kk$-linear dual
  $M^*:=\Hom_\kk(M,\kk)$ is a left $A$-module, i.\,e.\ a right
  $A^\opp$-module, and similarly the $\kk$-linear dual
  $\leftidx{^*}{N}{}$ of an $A^\opp$-module is an $A$-module.
  More precisely we have an adjunction of exact functors
  \begin{equation}
    \label{eq:k-dual-ModA}
    (-)^*
    \colon
    \Mod(A) \rightleftarrows (\Mod(A^\opp))^\opp
    \colon
    \leftidx{^*}{(-)}{}
  \end{equation}
  between abelian categories whose unit and counit
  are the obvious maps into the bidual.
  It induces an adjunction
  $\D(A) \rightleftarrows (D(A^\opp))^\opp$
  on (unbounded) derived categories which restricts to an adjoint
  equivalence
  \begin{equation}
    \label{eq:k-dual-DbmodA}
    (-)^*
    \colon
    \mathcal{T}=\D^\bd_{\mod(A)}(A)
    \rightleftarrows
    (\D^\bd_{\mod(A^\opp)}(A^\opp))^\opp
    \colon
    \leftidx{^*}{(-)}{}
  \end{equation}
  because any object of either category is isomorphic to a
  bounded complex of finite-dimensional modules over $A$ or
  $A^\opp$, respectively.

  Note that $A=\leftidx{_A}{A}{_A}$ is both a left $A$-module and
  a right $A$-module. Hence $\mathscr{D}:=\Hom_\kk(A,\kk) \in
  \mod(A \otimes_\kk A^\opp)$. This is the natural candidate
  bimodule to induce the equivalence~\eqref{eq:k-dual-DbmodA}. Let
  us check that it indeed induces this equivalence.
  If $M$ is a right $A$-module, then
  \begin{multline*}
    M^*=\Hom_\kk(M,\kk)
    =\Hom_\kk(M \otimes_A \leftidx{_A}{A}{_A}, \kk)\\
    =\Hom_A(M, \Hom_\kk(\leftidx{_A}{A}{_A}, \kk))=\Hom_A(M, \mathscr{D})
  \end{multline*}
  as left $A$-modules natural in $M$.
  Similarly, if $N$ is a left $A$-module,
  then
  \begin{multline*}
    \leftidx{^*}{N}{}=\Hom_\kk(N,\kk)
    =\Hom_\kk(\leftidx{_A}{A}{_A} \otimes_A N, \kk)\\
    =\Hom_{A^\opp}(N, \Hom_\kk(\leftidx{_A}{A}{_A}, \kk))
    =\Hom_{A^\opp}(N, \mathscr{D})
  \end{multline*}
  as right $A$-modules natural in $N$.
  Hence the functors $(-)^*$ and
  $\leftidx{^*}{(-)}{}$
  in the adjunction
  \eqref{eq:k-dual-ModA} and in the adjoint equivalence
  \eqref{eq:k-dual-DbmodA}
  may be written as
  \begin{equation*}
    (-)^* = \Hom_A(-,\mathscr{D})
    \qquad \text{and} \qquad
    \leftidx{^*}{(-)}{} = \Hom_{A^\opp}(-,\mathscr{D}).
  \end{equation*}
  Moreover, the unit and counit of the adjunction
  \eqref{eq:k-dual-ModA} correspond to the unit and counit
  of the adjunction obtained from $\mathscr{D}$, cf.\
  \eqref{eq:Hom-to-X-adjunction}.
  Since \eqref{eq:k-dual-DbmodA} is an adjoint equivalence,
  Remark~\ref{r:dualizing-object-from-duality} shows that
  $\mathscr{D}$ is a dualizing object for $\mathcal{T}$ and that
  \begin{equation}
    \label{eq:Tvee}
    \mathcal{T}^\vee=\D^\bd_{\mod(A^\opp)}(A^\opp)=\D^\bd(A^\opp)_\pseudocoh
  \end{equation}
  where the last equality comes from \eqref{eq:Db_modA-DbApscoh}.

  Now it is easy to check conditions
  \ref{enum:T-class-gen}
  and
  \ref{enum:dualizing-object-from-T-Tvee}
  from
  Theorem~\ref{t:sufficient-for-smoothness}
  in our situation.

  Condition~\ref{enum:T-class-gen} is obviously satisfied since
  $\mathcal{T}=\D^\bd_{\mod(A)}(A)\cong \D^\bd(\mod(A))$ has a
  classical generator, by
  Lemma~\ref{l:class-gen-Dbmod-fin-dim-alg}, and since the
  equalities \eqref{eq:T} and
  \eqref{eq:Tvee} hold.

  In order to check
  condition~\ref{enum:dualizing-object-from-T-Tvee},
  let $E$ be a classical generator of $\mathcal{T}$. Then its
  dual $E^*$ is a classical generator of $\mathcal{T}^{\vee}$.
  We may assume without loss of generality that $E$ and $E^*$ are
  bounded complexes of finite-dimensional modules.
  Since $\frac{A}{\rad(A)}$ is a separable $\kk$-algebra and the
  opposite of any separable algebra is separable we see that
  $\big(\frac{A}{\rad(A)}\big)^\opp=\frac{A^\opp}{\rad(A^\opp)}$
  is separable over $\kk$. Therefore
  $E \otimes_\kk E^*$ is a classical generator of
  $\D^\bd(\mod(A \otimes_\kk A^\opp))$, by
  Corollary~\ref{c:otimes-rad-separable-alg}, and also of the
  equivalent category
  $\D^\bd_{\mod(A \otimes_\kk A^\opp)}(A \otimes_\kk A^\opp)$.
  Since $\mathscr{D}$ obviously lies in this category we see that
  condition~\ref{enum:dualizing-object-from-T-Tvee} is satisfied;
  more precisely, we have checked the equivalent condition
  \ref{enum:dualizing-object-from-class-gen-T-Tvee}
  in Remark~\ref{r:on-t:sufficient-for-smoothness}.
  Now Theorem~\ref{t:sufficient-for-smoothness} applies and shows
  that $\D^\bd(\mod(A))$ is smooth over $\kk$.
\end{proof}

\begin{remark}
  \label{r:dgEndP-smooth}
  Let $A$ be a finite-dimensional algebra over a field $\kk$ such
  that $\frac{A}{\rad(A)}$ is separable over $\kk$.
  Then smoothness over $\kk$ of $\D^\bd(\mod(A))$ (see
  Theorem~\ref{t:Dbmod-findimalg-separable-smooth})
  has the following down-to-earth interpretation, by
  Remark~\ref{r:smoothness-dg-endos-classical-generator}:
  Let $S=\bigoplus_{i=1}^r S_i$ be a finite direct sum of simple
  $A$-modules such that each simple $A$-module is isomorphic to
  one of the $S_i$.
  Let $P$ be a projective resolution of $S$ in $\mod(A)$.
  Then the dg algebra $\ul\C_A(P,P)$ of endomorphisms of $P$ is
  $\kk$-smooth.
\end{remark}

\begin{remark}
  \label{r:finite-global-dim-smooth}
  Let $A$ be a finite-dimensional algebra over a field $\kk$
  such that $\frac{A}{\rad(A)}$ is separable over $\kk$ (this
  is automatic if $\kk$ is perfect), and
  assume that $A$ has finite global dimension.
  Then
  $A$ is a
  classical generator of $\D^\bd(\mod(A))$.  We have proven
  that this category is $\kk$-smooth, see
  Theorem~\ref{t:Dbmod-findimalg-separable-smooth}. This means
  that $A$ itself is $\kk$-smooth, by
  Remark~\ref{r:smoothness-dg-endos-classical-generator}.
  This also follows from \cite[Lemma~7.2]{rouquier-dimensions}.
\end{remark}

\section{Existence of a classical generator of
  \texorpdfstring{$\D^\bd(\coh(\mathcal{A}))$}{Db(coh A)}}
\label{sec:exist-class-gener}

The goal of this section is
Theorem~\ref{t:generator-DbcohA}: Given a coherent
$\mathcal{O}_X$-algebra $\mathcal{A}$ on a noetherian J-2 scheme $X$
(for example a scheme of finite type over a field or over
the integers),
the bounded derived category
$\D^\bd(\coh(\mathcal{A}))$
of coherent $\mathcal{A}$-modules has a classical generator.
We also prove Theorem~\ref{t:boxtimes-and-generators} which
says that the boxproduct of classical generators is a classical
generator for finite type schemes over a perfect field.

In contrast to our convention in the rest of this article, we
work with left modules in this section because this seems to be
the standard choice in commutative algebra and algebraic geometry.
By a \emph{noetherian} ring we mean a \emph{left noetherian} ring.

\subsection{Noncommutative structure sheaves}
\label{sec:nonc-struct-sheav}

Let $\mathcal{A}$ be a sheaf of (possibly noncommutative) rings on
a topological space. Quasi-coherent and coherent
(left) $\mathcal{A}$-modules are defined in the usual way
(cf.\ \cite[\sptag{01BE}, \sptag{01BV}]{stacks-project}
for modules over a sheaf of commutative rings).
Every coherent $\mathcal{A}$-module is quasi-coherent
(cf.\ \cite[\sptag{01BW}]{stacks-project}).
Let $\Qcoh(\mathcal{A})$ and $\coh(\mathcal{A})$ denote the
corresponding full
subcategories of the abelian category $\Mod(\mathcal{A})$ of all
$\mathcal{A}$-modules.
Recall that $\coh(\mathcal{A})$ is a full abelian subcategory of
$\Mod(\mathcal{A})$
(cf.\ \cite[\sptag{01BY}]{stacks-project}).

If $X$ is a scheme, the category $\Qcoh(\mathcal{O}_X)$ is a full
abelian subcategory of
$\Mod(\mathcal{O}_X)$
(see \cite[\sptag{01LA}]{stacks-project}).

\begin{lemma}
  \label{l:A-vs-O-qcoh-coh}
  Let $X$ be a scheme, $\mathcal{A}$ an $\mathcal{O}_X$-algebra
  (not necessarily commutative)
  and $M$ an $\mathcal{A}$-module.
  \begin{enumerate}
  \item
    \label{enum:Aqc=Oqc}
    Assume that $\mathcal{A}$ is
    $\mathcal{O}_X$-quasi-coherent. Then
    $M$ is $\mathcal{A}$-quasi-coherent if and only if $M$ is
    $\mathcal{O}_X$-quasi-coherent.
    In particular, $\Qcoh(\mathcal{A})$ is a full abelian
    subcategory of $\Mod(\mathcal{A})$.
  \item
    \label{enum:Ac=Oc}
    Assume that $X$ is locally noetherian and that $\mathcal{A}$ is
    $\mathcal{O}_X$-coherent. Then
    $M$ is $\mathcal{A}$-coherent if and only if $M$ is
    $\mathcal{O}_X$-coherent.
  \end{enumerate}
\end{lemma}

\begin{proof}
  \ref{enum:Aqc=Oqc}
  Assume that $M$ is $\mathcal{A}$-quasi-coherent. Then
  any $x \in X$ has an open neighborhood $U$ in $X$ such that
  there is an exact sequence
  $
  \bigoplus_J \mathcal{A}|_U
  \ra
  \bigoplus_I \mathcal{A}|_U
  \ra
  M|_U \ra 0
  $
  of $\mathcal{A}$-modules.
  Since $\mathcal{A}$ is $\mathcal{O}_X$-quasi-coherent,
  the first two terms of this sequence are
  $\mathcal{O}_U$-quasi-coherent. But then $M|_U$ is
  $\mathcal{O}_U$-quasi-coherent as a cokernel of a morphism
  between quasi-coherent $\mathcal{O}_U$-modules. This shows that
  $M$ is a quasi-coherent $\mathcal{O}_X$-module.

  Conversely, assume that $M$ is $\mathcal{O}_X$-quasi-coherent.
  Then any $x \in X$ has an open neighborhood
  $U$ in $X$ such that there is an epimorphism
  $
  \bigoplus_{I} \mathcal{O}_U \sra M|_U
  $
  of $\mathcal{O}_U$-modules.
  By adjunction we get an epimorphism morphism
  $
  \bigoplus_{I} \mathcal{A}|_U \sra M|_U
  $
  of $\mathcal{A}|_U$-modules. Let $N$ be its kernel, an
  $\mathcal{A}|_U$-module.
  Since $\bigoplus_{I} \mathcal{A}|_U$
  and $M|_U$ are $\mathcal{O}_X$-quasi-coherent, so is $N$.
  Repeating the above argument and possibly replacing $U$ by a
  smaller open neighborhood of $x$, we
  find an epimorphism
  $
  \bigoplus_{J} \mathcal{A}|_U \ra N
  $
  and hence an exact sequence
  $
  \bigoplus_{J} \mathcal{A}|_U \ra
  \bigoplus_{I} \mathcal{A}|_U \ra M|_U
  \ra 0
  $
  of $\mathcal{A}|_U$-modules.
  This shows that $M$ is $\mathcal{A}$-quasi-coherent.

  Since $\Qcoh(\mathcal{O}_X)$ is a full
  abelian subcategory of
  $\Mod(\mathcal{O}_X)$ we deduce that $\Qcoh(\mathcal{A})$ is a
  full abelian subcategory of $\Mod(\mathcal{A})$.

  \ref{enum:Ac=Oc}
  Assume that $M$ is $\mathcal{A}$-coherent.
  Then
  any $x \in X$ has an open neighborhood $U$ in $X$ such that
  there is an exact sequence
  $
  \bigoplus_{j=1}^m \mathcal{A}|_U
  \ra
  \bigoplus_{i=1}^n \mathcal{A}|_U
  \ra
  M|_U \ra 0
  $
  of $\mathcal{A}$-modules.
  Since the first two objects are
  $\mathcal{O}_U$-coherent and
  since $\coh(\mathcal{O}_U)$ is a full abelian subcategory of
  $\Mod(\mathcal{O}_U)$ we see that
  $M|_U$ is $\mathcal{O}_U$-coherent. This implies that $M$ is
  $\mathcal{O}_X$-coherent.

  Conversely, assume that $M$ is $\mathcal{O}_X$-coherent.
  Then $M$ is of finite type over $\mathcal{O}_X$ and a fortiori
  of finite type over $\mathcal{A}$. Let $U \subset X$ be open
  and let $\bigoplus_{i=1}^n \mathcal{A}|_U \ra
  M|_U$ be a morphism of $\mathcal{A}$-modules.
  Let $N$ be its kernel. Since $\bigoplus_{i=1}^n \mathcal{A}|_U$
  and $M|_U$ are $\mathcal{O}_U$-coherent, so is
  $N$. In particular $N$ is of finite type over $\mathcal{O}_U$
  and a fortiori of finite type over $\mathcal{A}$. This shows
  that $M$ is a coherent $\mathcal{A}$-module.
\end{proof}

\subsubsection{Inverse and direct image}
\label{sec:inverse-direct-image}

Let $f \colon Y \ra X$ be a morphism of schemes and let
$\mathcal{A}$ be an $\mathcal{O}_X$-algebra.
Then $\mathcal{B}:=f^*\mathcal{A}$ is an
$\mathcal{O}_Y$-algebra and
\begin{equation}
  \label{eq:f^*-adj-f_*}
  f^* \colon \Mod(\mathcal{A}) \rightleftarrows \Mod(\mathcal{B})
  \colon f_*
\end{equation}
is an adjunction where inverse image $f^*M= \mathcal{B}
\otimes_{f^{-1}\mathcal{A}} f^{-1}M$ and direct image $f_*$ are defined in the
usual way. This adjunction is compatible with the usual adjunction
\begin{equation*}
  f^* \colon \Mod(\mathcal{O}_X) \rightleftarrows \Mod(\mathcal{O}_Y)
  \colon f_*
\end{equation*}
in the sense that $f^*$ and $f_*$ commute with the forgetful
functors $\Mod(\mathcal{A}) \ra \Mod(\mathcal{O}_X)$ and
$\Mod(\mathcal{B}) \ra \Mod(\mathcal{O}_Y)$.

Assume that $f$ is quasi-compact and quasi-separated.
If $\mathcal{A}$ is $\mathcal{O}_X$-quasi-coherent, then
$\mathcal{B}$ is $\mathcal{O}_Y$-quasi-coherent and \eqref{eq:f^*-adj-f_*}
restricts to an adjunction
\begin{equation}
  \label{eq:f^*-adj-f_*-Qcoh}
  f^* \colon \Qcoh(\mathcal{A}) \rightleftarrows \Qcoh(\mathcal{B})
  \colon f_*
\end{equation}
by Lemma~\ref{l:A-vs-O-qcoh-coh}
because $f_*$ maps $\mathcal{O}_Y$-quasi-coherent modules to
$\mathcal{O}_X$-quasi-coherent modules by our assumptions on
$f$ (see
\cite[\sptag{01LC}]{stacks-project}).
If $X$ and $Y$ are locally noetherian and $\mathcal{A}$ is
$\mathcal{O}_X$-coherent, then $\mathcal{B}$ is
$\mathcal{O}_Y$-coherent and $f^*$ restricts to $f^* \colon
\coh(\mathcal{A}) \ra \coh(\mathcal{B})$,
by \cite[\sptag{01XZ}, \sptag{01BQ}]{stacks-project}
and Lemma~\ref{l:A-vs-O-qcoh-coh}. The direct image of
a coherent $\mathcal{B}$-module is in general not $\mathcal{A}$-coherent.

\subsubsection{The affine situation}
\label{sec:affine-situation}

Let $R$ be a commutative ring and $X=\Spec R$.
Serre's theorem states that taking global sections is an
equivalence
\begin{equation*}
  \Qcoh(\mathcal{O}_X) \sira \Mod(R)
\end{equation*}
between abelian categories which is compatible with tensor
products \cite[\sptag{01IB}, \sptag{01I8}]{stacks-project}.

Let $\mathcal{A}$ be a quasi-coherent $\mathcal{O}_X$-algebra and
$A=\mathcal{A}(X)$ the corresponding $R$-algebra.
Then an $A$-module corresponds under
Serre's equivalence to an
$\mathcal{O}_X$-quasi-coherent
$\mathcal{A}$-module which is, by
Lemma~\ref{l:A-vs-O-qcoh-coh}, the same thing as a quasi-coherent
$\mathcal{A}$-module. Hence we obtain an equivalence
\begin{equation*}
  \Qcoh(\mathcal{A}) \sira \Mod(A)
\end{equation*}
of abelian categories.

If $R$ is noetherian, Serre's equivalence restricts to an
equivalence
\begin{equation*}
  \coh(\mathcal{O}_X) \sira \mod(R)
\end{equation*}
where $\mod(R)$ is the category of finite (= finitely generated)
$R$-modules \cite[\sptag{01XZ}]{stacks-project}.
Let $\mathcal{A}$ be a coherent $\mathcal{O}_X$-algebra (= an
$\mathcal{O}_X$-algebra that is coherent as an
$\mathcal{O}_X$-module) and
$A=\mathcal{A}(X)$ the corresponding finite (and hence
noetherian) $R$-algebra (=
$R$-algebra that is finite as an $R$-module).
The same argument as above yields an equivalence
\begin{equation}
  \label{eq:cohA-modA}
  \coh(\mathcal{A}) \sira \mod(A)
\end{equation}
of abelian categories.

If $f \colon Y=\Spec S \ra X=\Spec R$ is a morphism of affine schemes,
the adjunction \eqref{eq:f^*-adj-f_*-Qcoh} corresponds to
the usual adjunction
\begin{equation*}
  B \otimes_A - \colon \Mod(A) \rightleftarrows \Mod(B)
  \colon \res^B_A
\end{equation*}
between extension and restriction of scalars along $A \ra B$
where $B:=(f^*\mathcal{A})(Y)= S \otimes_R A$.

\subsection{Verdier localization sequences and classical
  generators}
\label{sec:verd-local-sequ-1}

\begin{definition}
  \label{d:verdier-sequence}
  We say that a sequence
  \begin{equation*}
    \mathcal{S} \xra{i} \mathcal{T} \xra{q}
    \mathcal{Q}
  \end{equation*}
  of triangulated categories and functors is a
  \define{Verdier localization sequence} if the composition
  $q \circ i$ is zero, $i$ is fully faithful, and the induced
  functor from
  the Verdier quotient $\mathcal{T}/\im(i)$ to $\mathcal{Q}$
  is an equivalence where $\im(i)$ is the essential image of $i$.
\end{definition}

\begin{proposition}
  \label{p:verdier-seq-generators}
  Let $\mathcal{S} \xra{i} \mathcal{T} \xra{q} \mathcal{Q}$ be a
  Verdier localization sequence. If $\mathcal{S}$ and
  $\mathcal{Q}$ have classical generators, then $\mathcal{T}$
  has a classical generator.

  More precisely, if $E$ is a
  classical generator of $\mathcal{S}$ and $F$ is an object of
  $\mathcal{T}$ such that $q(F)$ is a classical generator of
  $\mathcal{Q}$, then $i(E) \oplus F$ is a classical generator of
  $\mathcal{T}$.
\end{proposition}

\begin{proof}
  We assume without loss of generality that $\mathcal{S}$ is a
  strictly full triangulated subcategory of $\mathcal{T}$ and
  that $q$ is the Verdier quotient functor $\mathcal{T} \ra
  \mathcal{T}/\mathcal{S}=\mathcal{Q}$.

  Recall from \cite[Prop.~II.2.3.1, items d), c)$^\text{bis}$,
  d)$^\text{bis}$]{verdier-these} that the
  obvious map defines a bijection between the set of thick
  subcategories of $\mathcal{T}$ containing $\mathcal{S}$ (and
  hence its thick closure in $\mathcal{T}$) and the set of
  thick subcategories of $\mathcal{T}/\mathcal{S} =
  \mathcal{Q}$.

  Let $\mathcal{U}$ be the thick subcategory of $\mathcal{T}$
  generated by $E \oplus F$. It contains $\mathcal{S}$ since $E$
  is a classical generator of $\mathcal{S}$.
  In order to show $\mathcal{U}=\mathcal{T}$ it is enough to see,
  by the above reminder,
  that the image of $\mathcal{U}$ in
  $\mathcal{Q}=\mathcal{T}/\mathcal{S}$ is all of
  $\mathcal{Q}$. But this is true because this image is a thick
  subcategory of
  $\mathcal{Q}$ that contains
  the classical generator $q(F)$.
\end{proof}

\subsection{Verdier localization sequence for  \texorpdfstring{$\D^\bd(\coh(\mathcal{A}))$}{Db(coh A)}}
\label{sec:verd-local-sequ}

If $X$ is a locally noetherian scheme and
$\mathcal{A}$ is a coherent $\mathcal{O}_X$-algebra, we
let $\D(\coh(\mathcal{A}))$ be the derived category of the
abelian category $\coh(\mathcal{A})$ of coherent $\mathcal{A}$
modules.
Its full subcategory
$\D^\bd(\coh(\mathcal{A}))$ of objects $M$
whose
total cohomology $\bigoplus_{n \in \bZ} \HH^n(M)$ is bounded is a
Karoubian subcategory (see \cite{le-chen-karoubi-trcat-bdd-t-str}).
If $Z \subset X$ is a closed subset let
$\D^\bd_Z(\coh(\mathcal{A}))$ denote the full subcategory of
$\D^\bd(\coh(\mathcal{A}))$ of objects whose cohomology sheaves
have support in the set $Z$. It is a thick subcategory.

\begin{theorem}
  \label{t:verdier-open-closed-cohA}
  Let $X$ be a
  noetherian scheme and
  $\mathcal{A}$ a
  coherent $\mathcal{O}_X$-algebra.
  Let $U$ be an open subscheme of $X$ and $Z :=X-U$ its
  closed complement.
  Then the sequence of triangulated categories
  \begin{equation}
    \label{eq:verdier-sequence}
    \D^\bd_Z(\coh(\mathcal{A}))
    \ra
    \D^\bd(\coh(\mathcal{A}))
    \ra
    \D^\bd(\coh(\mathcal{A}|_U))
  \end{equation}
  is a Verdier localization sequence where the first arrow is the
  inclusion and the second arrow is restriction to $U$.
\end{theorem}

\begin{proof}
  We abbreviate $\mathcal{A}_U:=\mathcal{A}|_U$.
  During the proof we assume without loss of generality that all
  objects
  of $\D^\bd(\coh(\mathcal{A}))$ are bounded complexes of
  coherent $\mathcal{A}$-modules, and
  similarly for
  $\D^\bd_Z(\coh(\mathcal{A}))$ and
  $\D^\bd(\coh(\mathcal{A}_U))$.

  Let $j \colon U \ra X$ be the open embedding.  Then
  $\mathcal{A}_U=j^*\mathcal{A}$ and
  $j^* \colon \coh(\mathcal{A}) \ra \coh(\mathcal{A}_U)$ is
  exact. We denote the induced functor
  \begin{equation*}
    j^* \colon \D^\bd(\coh(\mathcal{A})) \ra
    \D^\bd(\coh(\mathcal{A}_U)
  \end{equation*}
  by the same symbol.
  This functor is the second functor in \eqref{eq:verdier-sequence}.
  Clearly, its kernel is the
  subcategory $\D^\bd_Z(\coh(\mathcal{A}))$.  Let
  \begin{equation*}
    \mathcal{V}:=
    \frac{\D^\bd(\coh(\mathcal{A}))}{\D^\bd_Z(\coh(\mathcal{A}))}
  \end{equation*}
  be the Verdier quotient and
  \begin{equation*}
    \Phi \colon
    \mathcal{V}
    \ra
    \D^\bd(\coh(\mathcal{A}_U)
  \end{equation*}
  the induced triangulated functor. We need to prove that
  $\Phi$ is an equivalence. We first prove a useful fact.

  \textbf{Observation:}
  If $M$ is a bounded
  complex of quasi-coherent $\mathcal{A}$-modules
  whose restriction $j^*M$ consists of coherent $\mathcal{A}_U$-modules,
  then there is a subcomplex $K \subset M$ of coherent
  $\mathcal{A}$-modules such that
  $j^*K=j^*M$.

  Recall that every quasi-coherent $\mathcal{O}_X$-module is the
  directed colimit (or union) of
  its coherent submodules, see \cite[\sptag{01XZ},
  \sptag{01PG}]{stacks-project} where we use that $X$ is
  quasi-compact and quasi-separated
  \cite[\sptag{01OY}]{stacks-project}.
  The same statement is true for $\mathcal{A}$-modules: indeed,
  if $G$ is a coherent $\mathcal{O}_X$-submodule of a
  quasi-coherent $\mathcal{A}$-module $F$, then the image of
  $\mathcal{A} \otimes_{\mathcal{O}_X} G \ra F$ is a coherent
  $\mathcal{A}$-submodule of $F$ containing $G$.

  We deduce that every complex $M$ of quasi-coherent
  $\mathcal{A}$-modules is the directed colimit of its subcomplexes
  of coherent $\mathcal{A}$-modules: indeed, each component $M^n$
  is the directed colimit of its coherent
  $\mathcal{A}$-submodules, and if we are given coherent
  $\mathcal{A}$-submodules $N^n$ of $M^n$, for each $n \in \bZ$,
  there is a subcomplex of $M$ with coherent components which
  contains all $N^n$: just take $N^n+d(N^{n-1})$ in degree $n$.

  To prove the observation, let $M$ be a bounded complex of
  quasi-coherent
  $\mathcal{A}$-modules such that $j^*M$ has $\mathcal{A}_U$-coherent
  components. Write $M=\colim M_i$ as a directed colimit of
  subcomplexes $M_i$ of coherent $\mathcal{A}$-modules. Then
  $j^*M=j^*\colim M_i =\colim j^*M_i$. In particular, the $n$-th
  component
  $(j^*M)^n=\colim j^*(M_i^n)$ is
  a directed
  colimit of coherent $\mathcal{A}_U$-submodules and is
  itself $\mathcal{A}_U$-coherent by assumption. Hence
  $(j^*M)^n=j^*(M_i^n)$ for some $i$
  by \cite[\sptag{01Y8}]{stacks-project}
  and Lemma~\ref{l:A-vs-O-qcoh-coh}.
  Since $M$ is bounded
  there is some $i$ such that
  $(j^*M)^n=j^*(M_i^n)$ for all $n \in \bZ$, i.\,e.\ $M_i \subset
  M$ satisfies $j^*M_i=j^*M$.
  This proves the observation.

  In the following we will often use the adjunction
  \begin{equation*}
    j^* \colon \Qcoh(\mathcal{A}) \rightleftarrows
    \Qcoh(\mathcal{A}_U)
    \colon j_*
  \end{equation*}
  from \eqref{eq:f^*-adj-f_*-Qcoh} where we use the fact that
  $j$ is
  quasi-compact and quasi-separated
  as a map between noetherian schemes
  (see \cite[\sptag{02IK}, \sptag{01OY}, \sptag{01KV},
  \sptag{03GI}]{stacks-project}).
  Note that its counit is an isomorphism $j^*j_* \sira \id$.

  \textbf{$\Phi$ is essentially surjective:}
  Let $N$ be a bounded complex of coherent
  $\mathcal{A}_U$-modules.
  Then $j_*N$ is a bounded complex of quasi-coherent
  $\mathcal{A}$-modules which satisifes $j^*j_*N \sira N$.
  Hence our observation yields a subcomplex $K \subset j_*N$ with
  coherent components such that $j^*K = j^*j_*N \sira N$.
  This shows that $\Phi$ is essentially surjective.

  \textbf{$\Phi$ is faithful:}
  Let $g \colon M \ra M'$ be any morphism in
  $\mathcal{V}$.
  Then $g$ can be represented by a roof $M
  \xra{g'} M'' \xla{u} M'$ of morphisms in
  $\D^\bd(\coh(\mathcal{A}))$ where $u$ has cone in
  $\D^\bd_Z(\coh(\mathcal{A}))$, i.\,e.\ $g=u^{-1}g'$.
  Similarly, $g'$ can be represented by a roof
  $M \xra{g''} N \xla{u'} M''$ of morphisms in
  the homotopy category
  $\K^\bd(\coh(\mathcal{A}))$
  where $u'$ is a quasi-isomorphism and $N$ is a bounded complex
  of coherent $\mathcal{A}$-modules,
  i.\,e.\
  $g'=u'^{-1} g''$ in $\D^\bd(\coh(\mathcal{A}))$.
  Then $\Phi(g)=j^*(u)^{-1}j^*(g')=j^*(u)^{-1} j^*(u')^{-1}
  j^*(g'')$
  in $\D^\bd(\coh(\mathcal{A}_U))$.
  For faithfulness of $\Phi$ we need to prove that $\Phi(g)=0$
  implies $g=0$. Equivalently, we need to prove that $j^*(g'')=0$
  in $\D^\bd(\coh(\mathcal{A}_U))$ implies
  $g''=0$ in $\mathcal{V}$.

  Hence the proof of faithfulness of $\Phi$ is reduced to the
  following claim:
  Let $f \colon M \ra N$ be a morphism
  in the category
  $\C^\bd(\coh(\mathcal{A}))$
  of bounded complexes of
  coherent $\mathcal{A}$-modules such that $j^*(f)=0$
  in $\D^\bd(\coh(\mathcal{A}_U))$. Then
  $f=0$ in $\mathcal{V}$.

  The assumption $j^*(f)=0$ in $\D^\bd(\coh(\mathcal{A}_U))$ shows
  that the roof $j^*M \xra{j^*(f)} j^*N \xla{\id} j^*N$
  in the homotopy category
  $\K^\bd(\coh(\mathcal{A}_U))$
  is equivalent to the
  roof
  $j^*M \xra{0} j^*N \xla{\id} j^*N$. Hence there are a bounded complex
  $L$ of coherent $\mathcal{A}_U$-modules and a quasi-isomorphism
  $s \colon j^*N \ra L$
  in $\C^\bd(\coh(\mathcal{A}_U))$
  such that the composition
  \begin{equation*}
    j^*M \xra{j^*(f)} j^*N \xra{s} L
  \end{equation*}
  is homotopic to zero, i.\,e.\ there is a homotopy $h \colon
  j^*M \ra L[1]$ such that $s \circ j^*(f) = d(h)=dh+hd$.
  Let $s' \colon N \ra j_*L$ and $h' \colon M \ra j_*L[1]$
  correspond to $s$ and $h$ under the adjunction. Then the composition
  \begin{equation*}
    M \xra{f} N \xra{s'} j_*L
  \end{equation*}
  is homotopic to zero via the homotopy $h'$.

  Note that the image $s'(N) \subset j_*L$ is a subcomplex of coherent
  $\mathcal{A}$-modules. Similarly, $h'(M[-1]) \subset j_*L$ is a
  graded submodule of coherent $\mathcal{A}$-modules, and the
  subcomplex
  $h'(M[-1]) + d(h'(M[-1])) \subset j_*L$ it generates
  is a subcomplex of coherent $\mathcal{A}$-modules.
  Let $K \subset j_*L$ be a
  subcomplex of coherent $\mathcal{A}$-modules
  which contains these two subcomplexes and has the property that
  $j^*K=j^*j_*L$; it exists by the
  observation made above using
  $j^*j_*L\sira L$.

  By construction, $s'$ and $h'$ factor as $s' \colon N \xra{s''}
  K \hra j_*L$ and
  $h' \colon M \xra{h''} K[1] \hra j_*L[1]$, respectively, and the
  composition
  \begin{equation*}
    M \xra{f} N \xra{s''} K
  \end{equation*}
  is homotopic to zero via the homotopy $h''$, i.\,e.\
  $s'' \circ f=0$ in $\D^\bd(\coh(\mathcal{A}))$.

  Note that $s$ is the composition
  \begin{equation*}
    j^*N \xra{j^*(s'')} j^*K = j^*j_*L \sira L
  \end{equation*}
  of morphisms in $\C^\bd(\coh(\mathcal{A}))$.
  Since $s$ is a quasi-isomorphism, so is $j^*(s'')$. In
  particular, the mapping cone of $s''$ has
  cohomology supported in $Z$. Hence $s''$ becomes
  invertible in
  $\mathcal{V}$. Since
  $s'' \circ f=0$ in $\D^\bd(\coh(\mathcal{A}))$
  this implies $f=0$ in $\mathcal{V}$.
  This finishes the proof that $\Phi$ is faithful.

  \textbf{$\Phi$ is full:}
  Let $M, N$ be bounded complexes in $\coh(\mathcal{A})$. We need
  to show that any morphism
  $f \colon j^*M \ra j^*N$ in $\D^\bd(\coh(\mathcal{A}_U))$ comes
  from a morphism in
  $\mathcal{V}$.

  We first prove this statement under the more restrictive
  assumption that
  $f \colon j^*M \ra j^*N$ is a morphism in
  $\C^\bd(\coh(\mathcal{A}_U))$.

  Consider the diagram
  \begin{equation*}
    \xymatrix{
      {M} \ar[r]^-{\eta_M} & {j_*j^*M} \ar[d]^-{j_*(f)}
      \\
      {N} \ar[r]^-{\eta_N} & {j_*j^*N}
    }
  \end{equation*}
  in $\C^\bd(\Qcoh(\mathcal{A}))$ where the horizontal arrows
  are the respective adjunction units.
  The images of the morphisms $\eta_N$ and $j_*(f) \circ \eta_M$
  are
  subcomplexes
  of $j_*j^*N$ whose components are coherent $\mathcal{A}$-modules.
  Since $j^*j_*j^*N \sira j^*N$ is a complex of coherent
  $\mathcal{A}_U$-modules, there is a subcomplex $K \subset
  j_*j^*N$ consisting of coherent $\mathcal{A}$-modules which
  contains these two images and satisfies
  $j^*K = j^*j_*j^*N$; this follows
  from our observation.

  By construction we obtain the morphisms $\kappa$ and $\mu$ in
  the following left diagram turning it into a commutative
  diagram; the commutative diagram on the right is obtained from
  it by restriction to $U$ and by using formal properties of the
  adjunction $(j^*, j_*)$ where
  $\epsilon$ is the adjunction counit.
  \begin{equation*}
    \xymatrix{
      {M} \ar[rr]^-{\eta_M} \ar[rd]^-{\mu}
      && {j_*j^*M} \ar[d]^-{j_*(f)}
      \\
      {N} \ar@/_{2.5ex}/[rr]_-{\eta_N} \ar[r]^-{\kappa}
      & {K} \ar@{}[r]|-{\subset}
      & {j_*j^*N}
    }
    \qquad
    \xymatrix{
      {j^*M} \ar[rr]_-{j^*(\eta_M)}^-{\sim} \ar[rd]^(0.65){j^*(\mu)}
      \ar@/^{5ex}/[rrr]^-{\id}
      && {j^*j_*j^*M} \ar[d]^-{j^*j_*(f)} \ar[r]_-{\epsilon_{j^*M}}^-{\sim}
      & {j^*M} \ar[d]^-{f}
      \\
      {j^*N} \ar@/_{3ex}/[rr]_(0.6){j^*(\eta_N)}^(0.6){\sim} \ar[r]^-{j^*(\kappa)}
      \ar@/_{7ex}/[rrr]_-{\id}
      & {j^*K} \ar@{}[r]|-{=}
      & {j^*j_*j^*N} \ar[r]_-{\epsilon_{j^*N}}^-{\sim}
      & {j^*N}
    }
  \end{equation*}
  The diagram on the right shows
  $j^*(\kappa) \circ f= j^*(\mu)$ and
  that $j^*(\kappa)=j^*(\eta_N)$ is an
  isomorphism,
  so $\kappa$ becomes invertible in
  $\mathcal{V}$.
  Hence $\kappa^{-1} \circ \mu$ is a
  morphism in
  $\mathcal{V}$
  such that
  $\Phi(\kappa^{-1} \circ \mu)=(j^*(\kappa))^{-1} \circ j^*(\mu)=f$
  in $\D^\bd(\coh(\mathcal{A}_U))$.

  Now assume that $f \colon j^*M \ra j^*N$ is an arbitrary
  morphism in $\D^\bd(\coh(\mathcal{A}_U))$. It can be
  represented by a roof $j^*M \xra{g} j^*N' \xla{q} j^*N$ in
  $\C^\bd(\coh(\mathcal{A}_U))$ where $q$ is a quasi-isomorphism,
  i.\,e.\ $f=q^{-1} g$; here
  we can assume without loss of generality that the apex of
  our roof has the form $j^*N'$ where $N'$ is a bounded complex of coherent
  $\mathcal{A}$-modules, as follows from the proof of essential
  surjectivity of $\Phi$.

  As seen above, there are morphisms $\hat{g} \colon M
  \ra N'$ and $\hat{q} \colon N \ra N'$
  in $\mathcal{V}$
  such that
  $\Phi(\hat{g})=g$ and $\Phi(\hat{q})=q$.

  Note that $\Phi(\Cone(\hat{q}))\cong
  \Cone(\Phi(\hat{q})) \cong \Cone(q) = 0$ since $q$
  is an isomorphism in
  $\D^\bd(\coh(\mathcal{A}_U))$. Since we already know that
  $\Phi$ is faithful we get
  $\Cone(\hat{q}) \cong 0$ and hence $\hat{q}$ is an
  isomorphism. (Abstractly, we have used that a faithful
  triangulated functor reflects isomorphisms.)
  But then $\Phi(\hat{q}^{-1}g)=\Phi(\hat{q})^{-1}\Phi(g)=q^{-1}
  g=f$.
  This finishes the proof that $\Phi$ is full.
\end{proof}

\begin{remark}
  \label{r:serre-open-closed-cohA}
  In the setting of
  Theorem~\ref{t:verdier-open-closed-cohA},
  the sequence of abelian categories
  \begin{equation}
    \label{eq:serre-sequence}
    \coh_Z(\mathcal{A})
    \ra
    \coh(\mathcal{A})
    \ra
    \coh(\mathcal{A}|_U)
  \end{equation}
  is a \textit{Serre localization sequence} where
  $\coh_Z(\mathcal{A})$ is the full subcategory of
  $\coh(\mathcal{A})$ of objects with support in $Z$, the first
  arrow
  is the
  inclusion of this subcategory, and the second arrow is
  restriction to $U$.
  Here the term
  \textit{Serre localization sequence} means that
  the obvious functor from
  $\frac{\coh(\mathcal{A})}{\coh_Z(\mathcal{A})}$ to
  $\coh(\mathcal{A}|_U)$ is an equivalence.
  The proof is an obvious variation of the proof of
  Theorem~\ref{t:verdier-open-closed-cohA} and actually easier
  than this proof.
\end{remark}

\subsection{Classical generators and open-closed decompositions}
\label{sec:class-gener-open}

If $i \colon Z \ra X$ is the inclusion of a closed subscheme into
a noetherian scheme, then
$i_* \colon \Mod(\mathcal{O}_Z) \ra \Mod(\mathcal{O}_X)$ is exact
and preserves coherence. In particular, if $\mathcal{A}$ is a
coherent
$\mathcal{O}_X$-algebra, then
$i_* \colon \coh(i^*\mathcal{A}) \ra \coh(\mathcal{A})$ is
well-defined and exact
(cf.\ Lemma~\ref{l:A-vs-O-qcoh-coh}).
We use the same symbol for the induced functor
$i_* \colon \D^\bd(\coh(i^*\mathcal{A})) \ra
\D^\bd(\coh(\mathcal{A}))$.

\begin{proposition}
  \label{p:generator-support-Z}
  Let $Z$ be a closed subscheme of a noetherian scheme $X$ and
  let $i \colon Z \ra X$ be the inclusion.
  Let
  $\mathcal{A}$ be a coherent $\mathcal{O}_X$-algebra
  and set $\mathcal{A}_Z:=i^*\mathcal{A}$.
  Assume that $E$ is a classical generator of
  $\D^\bd(\coh(\mathcal{A}_Z))$. Then $i_*E$ is a classical
  generator of $\D^\bd_Z(\coh(\mathcal{A}))$.
\end{proposition}

\begin{proof}
  Note that $i_*E$ is an object of $\D^\bd_Z(\coh(\mathcal{A}))$.
  Let $\mathcal{U}$ be the thick subcategory of
  $\D^\bd_Z(\coh(\mathcal{A}))$ generated by $i_*E$.
  We need to show
  $\mathcal{U}=
  \D^\bd_Z(\coh(\mathcal{A}))$.

  Since any object of $\D^\bd_Z(\coh(\mathcal{A}))$ is build up
  from its finitely many non-zero cohomology
  modules, it is enough to show that
  any coherent $\mathcal{A}$-module $M$ with support in $Z$ is in
  $\mathcal{U}$. Note that $M$ is $\mathcal{O}_X$-coherent by
  Lemma~\ref{l:A-vs-O-qcoh-coh}.

  Let $\mathcal{I} \subset \mathcal{O}_X$ be the
  ($\mathcal{O}_X$-coherent)
  ideal sheaf of $Z$. Then there is some $n \in \bN$ such that
  $\mathcal{I}^n M=0$ (see \cite[\sptag{01Y9}]{stacks-project}).
  Hence $M$ has a finite filtration
  \begin{equation*}
    0=\mathcal{I}^n M \subset \mathcal{I}^{n-1} M \subset \dots
    \subset \mathcal{I}M \subset M
  \end{equation*}
  by coherent $\mathcal{A}$-modules. All subquotients
  $\mathcal{I}^iM/\mathcal{I}^{i+1}M$ are coherent
  $\mathcal{A}$-modules that are annihilated by $\mathcal{I}$.
  Hence all these subquotients are in $\mathcal{U}$ (cf.\
  proofs of \cite[\sptag{087T}, \sptag{01QY}]{stacks-project})
  and so is $M$.
\end{proof}

\begin{proposition}
  \label{p:glue-generator-open-closed}
  Let $X$ be a
  noetherian scheme and
  $\mathcal{A}$ a
  coherent $\mathcal{O}_X$-algebra.
  Let $U$ be an open subscheme of $X$ and let $Z$ be a closed
  subscheme of $X$ such that $Z :=X-U$ as sets.
  Let $\mathcal{A}_Z:=i^*\mathcal{A}$ where $i \colon Z \ra X$
  is the inclusion.
  If
  $\D^\bd(\coh(\mathcal{A}_Z))$
  and
  $\D^\bd(\coh(\mathcal{A}|_U))$
  each have a classical generator, then
  $\D^\bd(\coh(\mathcal{A}))$
  has a classical generator.

  More precisely,
  if $E$ is a classical generator of $\D^\bd(\coh(\mathcal{A}_Z))$
  and
  $F$ is a classical generator of $\D^\bd(\coh(\mathcal{A}|_U))$,
  then
  $i_*E \oplus \hat{F}$ is a classical generator
  of $\D^\bd(\coh(\mathcal{A}))$
  where $\hat{F}$
  is any object of
  $\D^\bd(\coh(\mathcal{A}))$ with $\hat{F}|_U \cong F$
  (such an object $\hat{F}$ exists by
  Theorem~\ref{t:verdier-open-closed-cohA}).
\end{proposition}

\begin{proof}
  Note that $i_*E$ is a classical generator of
  $\D^\bd_Z(\coh(\mathcal{A}))$ by
  Proposition~\ref{p:generator-support-Z}.
  Hence the proposition follows from
  Proposition~\ref{p:verdier-seq-generators}
  applied to the Verdier localization sequence from
  Theorem~\ref{t:verdier-open-closed-cohA}.
\end{proof}

\subsection{Local existence of a classical generator}
\label{sec:local-exist-class}

Given a ring $A$ (unital, associative, but not necessarily
commutative), we denote its center by $\Z(A)$.

\begin{definition}[{cf.\ e.\,g.\
    \cite[p.~95]{knus-ojanguren-descente-azumaya}, \cite[13.7.6]{McCRob}}]
  \label{d:azumaya}
  An \define{Azumaya algebra} (over its center) is a ring $A$
  that satisfies the following two conditions:
  \begin{enumerate}
  \item As a $\Z(A)$-module, $A$ is finitely generated projective.
  \item
    The ring map
    \begin{align*}
      \eta_A \colon A \otimes_{\Z(A)} A^\opp & \ra \End_{\Z(A)}(A),\\
      a \otimes b & \mapsto (x \mapsto axb),
    \end{align*}
    is an isomorphism
    where
    $\End_{\Z(A)}(A)$ is the ring of $\Z(A)$-module
    endomorphisms of $A$.
  \end{enumerate}
\end{definition}

We refer the reader to
\cite[Thm.~III.5.1]{knus-ojanguren-descente-azumaya}
or \cite[Thm.~2.1]{auslander-goldman-brauer} for equivalent
conditions characterizing Azumaya algebras; for example, a ring
is an Azumaya algebra if and only if it is separable as an
algebra over its center.

If $X$ is a scheme, we denote its regular locus by $X_\reg$.
\begin{definition}[{\cite[\sptag{07R3}]{stacks-project}}]
  \label{d:j2}
  A scheme $X$ is called \define{J-2} if it is locally noetherian
  and if for every morphism $Y \ra X$ locally of finite type the
  set $Y_\reg$ is open in $Y$.
\end{definition}

We refer the reader to \cite[\sptag{07R2}]{stacks-project} for
basic properties and examples of J-2 schemes. Important examples of J-2
schemes are schemes locally of finite type over a field or
schemes locally of finite type over the integers.

Recall that a commutative ring $R$ is called \textit{regular} if
it is noetherian and all localizations $R_\mfp$ at prime ideals
$\mfp$ are regular local rings (see
\cite[\sptag{00OD}]{stacks-project}).

The following proposition is the key for several later results
we prove by noetherian induction.

\begin{proposition}
  \label{p:locally-nilpotently-azumaya}
  Let $X$ be a non-empty J-2 scheme.
  Let $\mathcal{A}$ be a coherent $\mathcal{O}_X$-algebra (which is not
  assumed to be commutative).
  Then there exist a non-empty affine open
  subset $U$ of $X$ and a nilpotent two-sided ideal $I$ of
  $\mathcal{A}(U)$
  and
  $\mathcal{O}_X(U)$-algebras $A_1, \dots, A_r$, for some $r
  \in \bN$, such that each $A_i$ is an Azumaya algebra over its
  center $\Z(A_i)$, each center $\Z(A_i)$ is a regular
  ring, and
  there is an isomorphism
  \begin{equation*}
    \mathcal{A}(U)/I \cong A_1 \times \dots \times A_r
  \end{equation*}
  of $\mathcal{O}_X(U)$-algebras.
\end{proposition}

\begin{proof}
  Without loss of generality we may and will impose the following
  additional assumptions.
  \begin{enumerate}[label=(\roman*)]
  \item
    $X$ is irreducible.

    Indeed, any J-2 scheme is by definition locally noetherian,
    so we may
    assume that $X$ is noetherian. Then $X$ has only finitely many
    irreducible components.
    Choose one irreducible component and consider the
    complement in $X$ of the union of the other irreducible
    components. This is an irreducible open subscheme of $X$, and
    we may replace $X$ by this irreducible open subscheme and
    $\mathcal{A}$ by its restriction to this open subscheme.
  \item
    $X$ is affine, say $X=\Spec R$ where $R$ is a J-2
    ring.

    Indeed, just replace $X$ by a non-empty affine open subscheme
    (which is the spectrum of a J-2 ring, by
    \cite[\sptag{07R4}]{stacks-project})
    and $\mathcal{A}$ by its restriction to this open subscheme.
  \end{enumerate}
  Set $A = \mathcal{A}(X)$.
  Then $A$ is a finite algebra\footnote{Note that $A$ is not
    assumed to be commutative, but we extend the usual
    commutative algebra terminology in
    the obvious way: $A$ is an $R$-algebra which is finitely
    generated as an $R$-module.}
  over the noetherian ring $R$,
  and $\mathcal{A}$ is the corresponding coherent
  $\mathcal{O}_X$-algebra.
  In the following, we often
  work with $R$ and $A$ instead of $X$ and $\mathcal{A}$
  and use results from \ref{sec:affine-situation} without
  mentioning this explicitly. If we
  say that the statement of the proposition is true for the pair
  $(R,A)$ we mean that it is true for $\Spec R$ and the
  $\mathcal{O}_X$-algebra associated to $A$.

  Without loss of generality we may and will impose the following
  additional assumptions.
  \begin{enumerate}[resume,label=(\roman*)]
  \item $X=\Spec R$ is reduced (and hence integral), i.\,e.\ $R$
    is an integral
    domain.

    Indeed, let $\nil(R)$ be the nilradical of $R$. It is
    a nilpotent ideal because $R$ is
    noetherian. The two-sided ideal $\langle
    \nil(R)\rangle$ generated by
    $\nil(R)$ in $A$ is then also nilpotent.
    Hence, if the proposition is true for
    the reduced ring $R/\nil(R)$
    (corresponding to the underlying reduced scheme $X_\red=\Spec
    R/\nil(R)$)
    and the
    $R/\nil(R)$-algebra $A/\langle \nil(R) \rangle = A \otimes_R
    (R/\nil(R))$, then it is also true for $R$ and the $R$-algebra
    $A$.
  \item
    \label{enum:A-finite-free}
    $A$ is a finite free $R$-module.

    Indeed, by generic freeness,
    $X=\Spec R$ contains a non-empty open subscheme such
    that the restriction of $\mathcal{A}$ to this subscheme is
    finite free as a module over the structure sheaf
    \cite[Lemma~10.81]{goertz-wedhorn-AGI}. Hence there is a
    non-zero element $s \in R$ such that $A_s$ is a free
    $R_s$-module, and we can replace the pair $(R, A)$ by the
    pair $(R_s, A_s)$.
  \item
    \label{enum:R-subset-A}
    The structure morphism $R \ra A$ is injective, i.\,e.\ $R
    \subset A$.

    Indeed, all claims of the proposition are obvious if $A=0$.
    Otherwise, the structure morphism is injective since $A$ is a
    free module over the integral domain $R$.
  \end{enumerate}
  Let $K=\Quot(R)$ be the field of fractions of
  $R$. Geometrically, it is the stalk of the
  structure sheaf $\mathcal{O}_X$ at the generic point of
  the integral scheme $X=\Spec R$.
  Then $A_K:= A \otimes_R K$ is a finite-dimensional $K$-algebra.

  Note that $K$ is the localization of the integral domain $R$ at
  $S:=R-\{0\}$, that $A_K$ is the localization of $A$ at
  the central subset $S$, and that $A$ is torsion-free over
  $R$. Hence we
  have the following commutative diagram of inclusions.
  \begin{equation*}
    \xymatrix{
      {A}
      \ar@{}[r]|-{\subset}
      &
      {A_K}
      \\
      {R}
      \ar@{}[r]|-{\subset}
      \ar@{}[u]|-{\cup}
      &
      {K}
      \ar@{}[u]|-{\cup}
    }
  \end{equation*}
  Without loss of generality we may and will impose the following
  additional assumptions.
  \begin{enumerate}[resume,label=(\roman*)]
  \item $A_K$ is a semisimple $K$-algebra.

    Indeed, let
    $\rad(A_K)$ be the Jacobson radical of $A_K$. This is a
    nilpotent two-sided ideal of $A_K$, and $\frac{A_K}{\rad(A_K)}$ is a
    semisimple $K$-algebra.
    Moreover, $A \cap \rad(A_K)$ is a nilpotent two-sided ideal of $A$,
    and we may expand the above diagram to the following commutative diagram.
    \begin{equation*}
      \xymatrix{
        {A':=\frac{A}{A \cap \rad(A_K)}}
        \ar@{}[r]|-{\subset}
        &
        {\frac{A_K}{\rad(A_K)}}
        \\
        {A}
        \ar@{}[r]|-{\subset}
        \ar@{->>}[u]
        &
        {A_K}
        \ar@{->>}[u]
        \\
        {R}
        \ar@{}[r]|-{\subset}
        \ar@{}[u]|-{\cup}
        &
        {K}
        \ar@{}[u]|-{\cup}
      }
    \end{equation*}
    Note that $A'_K:=A' \otimes_R K = \frac{A_K}{\rad(A_K)}$ canonically
    as $K$-algebras.
    Hence, if the proposition is true for the pair $(R, A')$,
    it is also true for the pair $(R, A)$.
    Replacing $A$ by $A'$ may however destroy the assumption
    \ref{enum:A-finite-free}

    Before taking care of this, let us consider assumption
    \ref{enum:R-subset-A}.
    Since $R$ is an integral domain, its intersection with
    the nilpotent ideal $\rad(A_K)$ is $\{0\}$, so we may view $R$
    as a subring of $A'$, i.\,e.\ \ref{enum:R-subset-A} holds for
    the $R$-algebra $A'$ mutatis mutandis.

    Let us explain how to deal with
    assumption \ref{enum:A-finite-free}.
    Generic freeness provides a non-zero element $s \in R$ such
    that $(A')_s$ is a
    finite free $R_s$-module. Observe that
    \begin{equation*}
      (A')_s
      =
      \Big(\frac{A}{(A \cap \rad(A_K))}\Big)_s
      =
      \frac{A_s}{(A \cap \rad(A_K))_s}
    \end{equation*}
    and that $(A \cap \rad(A_K))_s=A_s \cap \rad(A_K)$ is a
    nilpotent two-sided ideal of $A_s$. Moreover,
    $((A')_s)_K=A'_K=\frac{A_K}{\rad(A_K)}$ is a semisimple
    $K$-algebra.

    Hence we may replace the pair $(R, A)$ without loss of
    generality by the pair $(R_s, (A')_s)$.
  \end{enumerate}
  Since $A_K$ is a finite-dimensional, semisimple $K$-algebra, it
  decomposes as a finite product
  \begin{equation}
    \label{eq:AK-product-simple}
    A_K = B_1 \times \dots \times B_r
  \end{equation}
  of finite-dimensional, simple $K$-algebras $B_1, \dots, B_r$,
  for some $r \in \bN$.

  Without loss of generality we may and will impose the following
  additional assumption.
  \begin{enumerate}[resume,label=(\roman*)]
  \item
    The decomposition~\eqref{eq:AK-product-simple} comes from a
    decomposition
    \begin{equation*}
      A = A_1 \times \dots \times A_r
    \end{equation*}
    of rings by extension of scalars along $R \hra K$, i.\,e.\
    $A_i \otimes_R K=B_i$ for all $i \in \{1, \dots, r\}$.

    Indeed, let $e_i$ be the central idempotent of $A_K$ such
    that $B_i=e_iA_K$. These idempotents satisfy $e_i e_j
    =\delta_{ij}e_i$ and
    $\sum e_i=1$.
    Then $e_i=\frac{a_i}{s_i}$ for some
    elements $a_i \in A$ and $s_i \in R-\{0\}$. Setting $s=s_1
    \cdots s_r\not=0$ we obtain
    the decomposition $A_s = e_1A_s \times \dots \times e_rA_s$ as
    rings. By construction, the scalar extension of this
    decomposition is
    the decomposition~\eqref{eq:AK-product-simple}.
    Hence we may replace the pair $(R,A)$ by the pair $(R_s,
    A_s)$.
  \end{enumerate}
  Note that $\Z(A) = \Z(A_1) \times \dots \Z(A_r)$ and hence
  geometrically
  \begin{equation*}
    \Spec \Z(A) = \bigsqcup \Spec \Z(A_i) \ra \Spec R.
  \end{equation*}
  The $\Z(A)$-module $A$, viewed as a coherent module on
  $\Spec(\Z(A))$, decomposes
  accordingly: its restriction to each $\Spec \Z(A_i)$ is the
  $Z(A_i)$-module $A_i$.

  Since $A$ is a finite algebra over the noetherian ring $R$, its
  center $Z(A)$ is a finite commutative $R$-algebra. Since $R$ is
  J-2, the regular locus $(\Spec Z(A))_\reg$ is open in $\Spec
  Z(A)$, and
  $(\Spec Z(A_i))_\reg$ is open in $\Spec Z(A_i)$, for each $i
  \in \{1, \dots, r\}$.

  Without loss of generality we may and will impose the following
  additional assumption.
  \begin{enumerate}[resume,label=(\roman*)]
  \item
    The center $\Z(A_i)$ of $A_i$ is a regular ring and
    $A_i$ is a finite free $\Z(A_i)$-module,
    for all $i=\{1, \dots, r\}$.

    Indeed, since the finite free $R$-module $A$ is torsion-free
    over $R$, we have
    \begin{equation}
      \label{eq:ZA}
      \Z(A) \subset \Z(A) \otimes_R K = \Z(A \otimes_R K)
      =\Z(A_K)
    \end{equation}
    and hence
    $Z(A_i) \subset Z(B_i)$
    for each $i$. We fix
    $i$ for a moment.
    Since $B_i$ is a matrix ring over a division ring,
    its center
    $\Z(B_i)$ is a field, so its subring $\Z(A_i)$ is an integral
    domain.
    In particular, the generic point of $\Spec Z(A_i)$ is
    regular. This shows that
    the open regular locus $(\Spec Z(A_i))_\reg$ of $\Spec Z(A_i)$ is
    non-empty.

    By generic freeness, there is a non-empty open subset $V_i$
    of $(\Spec Z(A_i))_\reg$ such that $A_i|_{V_i}$ is a finite
    free $\Z(A_i)$-module.
    Let $f \colon \Spec \Z(A_i) \ra \Spec R$ be the finite
    morphism corresponding to the finite ring morphism
    $R \ra \Z(A_i)$ and consider the proper closed subset
    $C_i:=(\Spec \Z(A_i)) - V_i$ of $\Spec \Z(A_i)$.
    Since finite morphisms are closed
    and the
    generic point of $\Spec \Z(A_i)$ is the only point of
    $\Spec \Z(A_i)$ whose image under $f$ is the generic point of
    $\Spec R$, by \cite[Cor.~5.9]{atiyah-macdonald},
    the set $f(C_i)$ is a proper closed subset of $\Spec R$.
    Choose $s_i \in R-\{0\}$ such that $\Spec R_{s_i} \subset
    (\Spec R)-f(C_i)$.
    Then
    $\Z(A_i)_{s_i}$ is a regular ring
    and
    $(A_i)_{s_i}$ is a
    finite free $\Z(A_i)_{s_i}$-module. We find such an element
    $s_i$ for all $i \in \{1, \dots, r\}$.

    Set $s=s_1 \cdots s_r \in R-\{0\}$ and observe
    that $\Z(A_s)= (\Z(A))_s$ and $\Z((A_i)_s)=\Z(A_i)_s$.
    Hence, by replacing $R$ by
    $R_s$ and $A$ by $A_s$
    we can
    and will assume (in addition to the assumptions already
    imposed) that
    $\Z(A_i)$ is a regular ring and that $A_i$ is a finite free
    $\Z(A_i)$-module, for each $i \in \{1, \dots, r\}$.
  \end{enumerate}

  For any $i \in \{1,\dots, r\}$ consider the morphism
  \begin{equation*}
    \eta_{A_i} \colon A_i \otimes_{\Z(A_i)} A_i^\opp \ra \End_{\Z(A_i)}(A_i).
  \end{equation*}
  Its scalar extension
  \begin{equation*}
    \eta_{A_i} \otimes_R K \colon (A_i \otimes_{\Z(A_i)} A_i^\opp) \otimes_R
    K \ra \End_{\Z(A_i)}(A_i) \otimes_R K
  \end{equation*}
  identifies (note that $A_i$ is a finite free $\Z(A_i)$-module)
  via the canonical isomorphisms with
  \begin{equation*}
    \eta_{A_i \otimes_R K}
    \colon
    (A_i \otimes_R K)
    \otimes_{\Z(A_i) \otimes_R K} (A_i \otimes_R K)^\opp
    \ra \End_{\Z(A_i) \otimes_R K}(A_i \otimes_R K).
  \end{equation*}
  Using $A_i \otimes_R K = B_i$ and
  $\Z(A_i) \otimes_R K = \Z(B_i)$ (obtained from \eqref{eq:ZA})
  this morphism is the morphism
  \begin{equation*}
    \eta_{B_i} \colon B_i \otimes_{\Z(B_i)} B_i^\opp \ra \End_{\Z(B_i)}(B_i).
  \end{equation*}
  Since $B_i$ is a matrix ring with entries in a division ring $D_i$
  that has finite dimension over the field $\Z(B_i)=\Z(D_i)$, $B_i$ is a
  separable $\Z(B_i)$-algebra, i.\,e.\ an Azumaya algebra (see
  \cite[Thms.~III.3.1 and III.5.1]{knus-ojanguren-descente-azumaya}).
  This means that $\eta_{B_i}=\eta_{A_i \otimes_R K}$ is an isomorphism.

  Note that source and target of $\eta_{A_i}$ are finite free
  $\Z(A_i)$-modules.
  In particular, $\eta_{A_i}$ may be viewed as a morphism of
  coherent $\mathcal{O}_{\Spec A}$-modules which is an
  isomorphism at the generic point. Hence it is already an
  isomorphism on an open neighborhood of the generic point,
  i.\,e.\
  there is an element $s_i \in R-\{0\}$ such that
  $\eta_{A_i} \otimes_R R_{s_i}$ is an isomorphism.
  Since we already know that
  $(A_i)_{s_i}$ is a finite free module over
  $\Z((A_i)_{s_i})=\Z(A_i)_{s_i}$ we see that $(A_i)_{s_i}$ is an
  Azumaya algebra over its center.

  As above, we may set $s=s_1, \dots, s_r$ and replace
  $R$ by $R_s$ and $A$ by $A_s$ without destroying our previous
  assumptions.
  Then $A=A_1 \times \dots \times
  A_r$ as $R$-algebras where each $A_i$ is an Azumaya algebra
  over its center $\Z(A_i)$ which is a regular ring.
  This finishes the proof of the proposition.
\end{proof}

\begin{lemma}
  \label{l:Azumaya-regular-center-pdim}
  Let $A$ be an Azumaya algebra whose center $\Z(A)$ is a regular
  ring. If $M$ is any finitely generated
  $A$-module then its projective dimension $\pdim_A M$ is finite.
\end{lemma}

\begin{proof}
  Let $M$ be a finitely generated $A$-module.
  Then
  $\pdim_A M \leq \pdim_{Z(A)} M$
  by
  \cite[Thms.~1.8 and 2.1]{auslander-goldman-brauer} (with $\Delta=R$
  there).
  Note that $M$ is a finitely generated $\Z(A)$-module since, by
  definition of an Azumaya algebra, $A$
  is a finitely generated $\Z(A)$-module.
  Now observe that this implies
  $\pdim_{Z(A)} M < \infty$ since $Z(A)$ is a regular
  ring
  (as explained in the proof of
  \cite[Prop.~A.2]{daniel-valery-olaf-geometricity}).
\end{proof}

\begin{remark}
  There are regular
  commutative rings of
  infinite global dimension which have the property that every
  finitely generated module has finite projective dimension; an
  example is due to Nagata, see
  \cite[Exercise~11.4]{atiyah-macdonald},
  \cite[Example~7.7.2]{McCRob}.
\end{remark}

\begin{lemma}
  \label{l:classical-generator-DbmodA-Azumaya-quotient}
  Let $A$ be a noetherian ring (not assumed to be
  commutative). If $I \subset A$ is a nilpotent two-sided
  ideal such that
  each finitely generated $A/I$-module has
  finite projective
  dimension over $A/I$, then $A/I$ is a classical generator
  of $\D^\bd(\mod(A))$.
\end{lemma}

\begin{proof}
  Let $\mathcal{S}$ be the thick
  subcategory of $\D^\bd(\mod(A))$ generated by
  $A/I$.
  By assumption, every finitely generated
  $A/I$-module has
  a finite
  resolution by finitely generated projective $A/I$-modules.
  Hence every finitely generated $A$-module that
  is annihilated by $I$ is contained in $\mathcal{S}$.

  Let $M$ be any finitely generated
  $A$-module.
  Since $I$ is nilpotent, say $I^n=0$ for some $n \in \bN$,
  $M$ has a finite filtration
  \begin{equation*}
    0=I^n M \subset I^{n-1}M \subset \dots \subset IM \subset M
  \end{equation*}
  by submodules.
  Each subquotient $I^iM/I^{i+1}M$ is annihilated by $I$ and
  finitely generated as a module over the noetherian ring
  $A$. Hence $I^iM/I^{i+1}M \in \mathcal{S}$ for all
  $i$ and therefore $M \in \mathcal{S}$.

  Since any object of $\D^\bd(\mod(A))$ is built up
  from its finitely many non-zero cohomology modules, which are
  finitely generated
  $A$-modules, we deduce that
  $\D^\bd(\mod(A))=\mathcal{S}$.
\end{proof}

\begin{proposition}
  \label{p:local-generator}
  Let $X$ be a non-empty J-2 scheme.
  Let $\mathcal{A}$ be a
  coherent $\mathcal{O}_X$-algebra (which is not
  assumed to be commutative).
  Then there exists a non-empty affine open
  subset $U$ of $X$
  and a nilpotent two-sided ideal sheaf $\mathcal{I} \subset
  \mathcal{A}|_U$
  such that
  $\mathcal{A}|_U/\mathcal{I}$ is a classical generator of
  $\D^\bd(\coh(\mathcal{A}|_U))$.
\end{proposition}

\begin{proof}
  Proposition~\ref{p:locally-nilpotently-azumaya} provides
  a non-empty open subset $U$ of $X$ and a nilpotent two-sided
  ideal $I$ of $\mathcal{A}(U)$ such that
  $\mathcal{A}(U)/I$ is isomorphic to a finite product of Azumaya
  algebras whose centers are regular.
  By
  Lemma~\ref{l:Azumaya-regular-center-pdim}, any finitely
  generated $\mathcal{A}(U)/I$-module has finite projective
  dimension.
  Lemma~\ref{l:classical-generator-DbmodA-Azumaya-quotient}
  therefore shows that $\mathcal{A}(U)/I$ is a classical
  generator
  of
  $\D^\bd(\mod(\mathcal{A}(U)))$.
  Now transfer this statement to
  $\D^\bd(\coh(\mathcal{A}|_U))$ using the equivalence
  $\coh(\mathcal{A}|_U) \cong \mod(\mathcal{A}(U))$ (cf.\
  \eqref{eq:cohA-modA}).
\end{proof}

\subsection{Global existence of a classical generator}
\label{sec:glob-exist-class}

\begin{theorem}
  \label{t:generator-DbcohA}
  Let $X$ be a noetherian J-2 scheme and
  $\mathcal{A}$ a coherent $\mathcal{O}_X$-algebra. Then
  $\D^\bd(\coh(\mathcal{A}))$ has a classical generator.
\end{theorem}

\begin{proof}
  By noetherian induction we may assume that the
  category $\D^\bd(\coh(\mathcal{A}|_Z))$ has a classical
  generator
  for all proper
  closed subschemes $Z$ of $X$; note that any such $Z$ is again
  noetherian J-2.
  Obviously, we may assume that $X\not=\emptyset$.

  Proposition~\ref{p:local-generator} yields a non-empty open
  subset $U$ of $X$ such that $\D^\bd(\coh(\mathcal{A}|_U))$ has
  a classical generator.
  Equip $Z:=X - U$ with the reduced scheme structure. By
  noetherian induction we know that
  $\D^\bd(\coh(\mathcal{A}|_Z))$ has a classical generator.
  Proposition~\ref{p:glue-generator-open-closed} then shows that
  $\D^\bd(\coh(\mathcal{A}))$ has a classical generator.
\end{proof}

\begin{remark}
  Theorem~\ref{t:generator-DbcohA}
  shows in particular that $\D^\bd(\coh(\mathcal{O}_X))$ has a
  classical generator if $X$ is a noetherian J-2 scheme.
  We refer the reader to
  \cite[Thm.~7.38]{rouquier-dimensions}
  (concerning strong generation) and
  \cite[Prop.~6.8]{lunts-categorical-resolution}
  for related statements for separated schemes of finite type
  over a field. There are more general recent results
  by
  Neeman concerning strong generation (see
  \cite{neeman-strong-generators}).
\end{remark}

\begin{remark}
  \label{r:generator-cohA}
  In the setting of Theorem~\ref{t:generator-DbcohA},
  the category
  $\coh(\mathcal{A})$ even has a
  generator as defined in
  \cite[2.3]{iyengar-takahashi-openness-2018}
  (this
  certainly implies Theorem~\ref{t:generator-DbcohA}).
  The proof of this result is a straightforward
  variation of the results leading to
  Theorem~\ref{t:generator-DbcohA}; instead of working on the
  triangulateted level one works on the abelian level and
  uses the
  Serre localization sequence
  \eqref{eq:serre-sequence}
  instead of the Verdier localization sequence
  \eqref{eq:verdier-sequence}.
\end{remark}

\subsection{Boxproduct of classical generators}
\label{sec:boxpr-class-gener}

If $X$ and $Y$ are schemes over a field $\kk$ we write $X \times
Y$ instead of $X \times_\kk Y$. If $E$ is an
$\mathcal{O}_X$-module and $F$ is an
$\mathcal{O}_Y$-module we abbreviate $E \boxtimes
F:= p^*E \otimes_{\mathcal{O}_{X \otimes Y}}
q^*F$ where $X \xla{p} X \times Y \xra{q} Y$ are the projections.

If $X$ and $Y$ are affine, say $X=\Spec R$ and $Y=\Spec S$, and
$E$ and $F$ are quasi-coherent, then
$E \boxtimes F$ corresponds to the $R \otimes_\kk S$-module $E(X)
\otimes_\kk F(X)$.

If $\mathcal{A}$ is an $\mathcal{O}_X$-algebra and $\mathcal{B}$
is an $\mathcal{O}_Y$-algebra, then $\mathcal{A} \boxtimes
\mathcal{B}$ is a an $\mathcal{O}_{X \times Y}$-algebra.
Given an $\mathcal{A}$-module $E$ and a
$\mathcal{B}$-module $F$, then
$E \boxtimes F$ is in the obvious way an $\mathcal{A} \boxtimes
\mathcal{B}$-module. If $X$ and $Y$ are affine and $\mathcal{A}$,
$\mathcal{B}$, $E$ and $F$ are quasi-coherent over the structure
sheaves $\mathcal{O}_X$ and $\mathcal{O}_Y$-respectively, then
the $\mathcal{A} \boxtimes \mathcal{B}$-module $E \boxtimes
F$
corresponds to
the $\mathcal{A}(X) \otimes_\kk \mathcal{B}(Y)$-module $E(X)
\otimes_\kk F(Y)$.

\begin{theorem}
  \label{t:boxtimes-and-generators}
  Let $X$ and $Y$ be noetherian J-2 schemes over a perfect field
  $\kk$ such that $X \times Y$ is Noetherian (these assumptions
  are for example satisfied if $X$
  and $Y$ are schemes of finite type over the perfect field
  $\kk$).
  Let $\mathcal{A}$ be a coherent $\mathcal{O}_X$-algebra
  and let $\mathcal{B}$ be a coherent $\mathcal{O}_Y$-algebra.
  Consider the coherent $\mathcal{O}_{X \times Y}$-algebra
  $\mathcal{A} \boxtimes \mathcal{B}$.
  Then the following two statements are true:
  \begin{enumerate}
  \item
    \label{enum:exist-boxtimes}
    There exist a
    classical generator $E$ of $\D^\bd(\coh(\mathcal{A}))$ and a
    classical generator $F$ of $\D^\bd(\coh(\mathcal{B}))$ such that
    $E \boxtimes F$ is a classical generator of
    $\D^\bd(\coh(\mathcal{A} \boxtimes \mathcal{B}))$.
  \item
    \label{enum:all-boxtimes}
    For each
    classical generator $E$ of $\D^\bd(\coh(\mathcal{A}))$ and each
    classical generator $F$ of $\D^\bd(\coh(\mathcal{B}))$ the object
    $E \boxtimes F$ is a classical generator of
    $\D^\bd(\coh(\mathcal{A} \boxtimes \mathcal{B}))$.
  \end{enumerate}
\end{theorem}

\begin{proof}
  The obvious analog of the argument used at the beginning of the
  proof of
  Corollary~\ref{c:otimes-rad-separable-alg} shows that
  \ref{enum:exist-boxtimes} implies
  \ref{enum:all-boxtimes}.

  We prove \ref{enum:exist-boxtimes} in several steps.
  Proposition~\ref{p:locally-nilpotently-azumaya} will play a key
  role in the proof and motivates the following ad hoc
  terminology.

  A pair $(U, \mathcal{R})$
  consisting of an affine J-2 scheme $U$ over $\kk$ and a
  coherent $\mathcal{O}_U$-algebra $\mathcal{R}$ is called \textit{nice}
  if there is a nilpotent two-sided ideal $I \subset
  \mathcal{R}(U)$ such that
  \begin{equation*}
    \mathcal{R}(U)/I \cong A_1 \times \dots \times A_r
  \end{equation*}
  as $\mathcal{O}_U(U)$-algebras for suitable Azumaya algebras
  $A_i$
  whose centers $\Z(A_i)$ are regular rings.
  Note that each $\Z(A_i)$ is a Noetherian ring since it is a
  finite algebra over the noetherian ring $\mathcal{O}_U(U)$.

  Given a nice pair $(U, \mathcal{R})$ we do not distinguish
  between coherent $\mathcal{R}$-modules and finitely generated
  $\mathcal{R}(U)$-modules (cf.\ \eqref{eq:cohA-modA}).

  \textbf{Claim 1.} Statement
  \ref{enum:exist-boxtimes} is true if $(X, \mathcal{A})$ and
  $(Y, \mathcal{B})$ are nice.

  Let $I \subset \mathcal{A}(X)$ and $J \subset \mathcal{B}(Y)$
  be nilpotent two-sided ideals such that
  \begin{align*}
    \mathcal{A}(X)/I & \cong A_1 \times \dots \times A_r,\\
    \mathcal{B}(Y)/J & \cong B_1 \times \dots \times B_s
  \end{align*}
  where all $A_i$ and $B_j$ are Azumaya algebras with regular
  center.
  Then, by
  Lemma~\ref{l:classical-generator-DbmodA-Azumaya-quotient},
  $\mathcal{A}(X)/I$ and
  $\mathcal{B}(Y)/J$ are classical generators of
  $\D^\bd(\mod(\mathcal{A}(X)))$ and
  $\D^\bd(\mod(\mathcal{B}(Y)))$, respectively.

  Note that
  \begin{equation*}
    \frac{\mathcal{A}(X)}{I} \otimes_\kk \frac{\mathcal{B}(Y)}{J}
    \cong \prod_{i,j} A_i \otimes_\kk B_j.
  \end{equation*}
  Each factor $A_i \otimes_\kk B_j$ is an Azumaya algebra over
  its center
  \begin{equation*}
    \Z(A_i \otimes_\kk B_j)=\Z(A_i) \otimes_\kk \Z(B_j),
  \end{equation*}
  by \cite[Prop.~1.5]{auslander-goldman-brauer} (for
  $R_1=\Z(A_i)$, $R_2=\Z(B_j)$ and $R=\kk$)
  and \cite[Thm.~III.5.1]{knus-ojanguren-descente-azumaya}), and
  this center $\Z(A_i \otimes_\kk B_j)$ is a regular ring because
  it is noetherian as a finite algebra over the noetherian ring
  $\mathcal{O}_X(X) \otimes_\kk \mathcal{O}_Y(Y)=\mathcal{O}_{X
    \times Y}(X \times Y)$ and the field
  $\kk$ is perfect, so we obtain regularity from
  \cite[Thm.~6.(e)]{tousi-yassemi-tensor-product-of-rings}.
  Note that $I \otimes_\kk \mathcal{B}(Y)
  +
  \mathcal{A}(X) \otimes_\kk J$ is a nilpotent two-sided
  ideal in
  $\mathcal{A}(X) \otimes_\kk \mathcal{B}(Y)$ with quotient
  \begin{equation*}
    \frac{
      \mathcal{A}(X) \otimes_\kk \mathcal{B}(Y)
    }
    {
      I \otimes_\kk \mathcal{B}(Y)
      +
      \mathcal{A}(X) \otimes_\kk J
    }
    =
    \frac{\mathcal{A}(X)}{I} \otimes_\kk \frac{\mathcal{B}(Y)}{J}.
  \end{equation*}
  These facts imply, by
  Lemma~\ref{l:classical-generator-DbmodA-Azumaya-quotient},
  that
  $\frac{\mathcal{A}(X)}{I} \otimes_\kk \frac{\mathcal{B}(Y)}{J}$
  is a classical generator of
  $\D^\bd(\mod(\mathcal{A}(X) \otimes_\kk \mathcal{B}(Y)))$.
  Using the above description of the $\boxtimes$-product
  in terms of modules
  if $X$ and $Y$ are affine, this proves claim 1.

  \textbf{Claim 2.} Statement
  \ref{enum:exist-boxtimes} is true if at least one of
  $(X, \mathcal{A})$ and
  $(Y, \mathcal{B})$ is nice.

  Assume without loss of generality that $(Y, \mathcal{B})$ is nice.
  By noetherian induction on $X$ we can assume that
  for all proper closed subschemes $Z$ of $X$ with inclusion
  morphism $i \colon Z \subset X$
  there is a classical generator $E_Z$ of
  $\D^\bd(\coh(i^*\mathcal{A}))$ and a classical generator $F$ of
  $\D^\bd(\coh(\mathcal{B}))$ such that $E_Z \boxtimes F$ is a
  classical generator of
  $\D^\bd(\coh(i^*\mathcal{A} \boxtimes
  \mathcal{B}))$. Here we implicitly use that any such closed
  subscheme $Z$ is
  noetherian J-2 and has the property that $Z \times Y$ is noetherian.

  By assumption, $X$ is a noetherian J-2 scheme.
  We may assume that $X$ is non-empty.
  Then Proposition~\ref{p:locally-nilpotently-azumaya}
  yields a non-empty affine open
  subset $U$ of $X$ such that $(U, \mathcal{A}|_U)$ is nice; this
  uses that any affine open subscheme of a J-2 scheme is
  J-2. Note also that $U \times Y$ is noetherian.

  By claim 1 there are a classical
  generator $E_U$ of $\D^\bd(\coh(\mathcal{A}|_U))$ and a
  classical generator $F$ of $\D^\bd(\coh(\mathcal{B}))$ such that
  $E_U \boxtimes F$ is a classical generator of
  $\D^\bd(\coh(\mathcal{A}|_U \boxtimes \mathcal{B}))$.

  Equip $Z:=X - U$ with the reduced scheme structure and let $i
  \colon Z \subset X$ be the inclusion morphism. By
  noetherian induction (and the observation that
  \ref{enum:exist-boxtimes}
  implies \ref{enum:all-boxtimes}) there is a classical generator
  $E_Z$ of
  $\D^\bd(\coh(i^*\mathcal{A}))$ such that $E_Z \boxtimes F$ is a
  classical generator of
  $\D^\bd(\coh(i^*\mathcal{A} \boxtimes \mathcal{B}))$.

  By Theorem~\ref{t:verdier-open-closed-cohA} there is an object
  $\hat{E}_U$ of $\D^\bd(\coh(\mathcal{A}))$ such that
  $\hat{E}_U|_U\cong E_U$.
  Then $i_*E_Z \oplus \hat{E}_U$ is a classical generator of
  $\D^\bd(\coh(\mathcal{A}))$, by
  Proposition~\ref{p:glue-generator-open-closed}.
  The same proposition shows that
  \begin{equation*}
    (i_*E_Z \oplus \hat{E}_U) \boxtimes F
    \cong
    (i_*E_Z \boxtimes F) \oplus (\hat{E}_U \boxtimes F)
    \cong
    (i \times \id)_*(E_Z \boxtimes F) \oplus (\hat{E}_U \boxtimes F)
  \end{equation*}
  is a classical generator of
  $\D^\bd(\coh(\mathcal{A} \boxtimes \mathcal{B}))$
  because
  $(\hat{E}_U \boxtimes F)|_U \cong E_U \boxtimes F$.
  This proves the claim.

  \textbf{Claim 3.} Statement
  \ref{enum:exist-boxtimes} is true for arbitrary
  $(X, \mathcal{A})$ and
  $(Y, \mathcal{B})$.

  To prove this we proceed as in the proof of claim 2 but do of
  course not assume
  that $(Y, \mathcal{B})$ is nice; the only other
  difference is that we invoke claim 2 at the place where we invoke
  claim 1 in the proof of claim 2.
\end{proof}

\section{Smoothness for some algebras which are finite over their
  center}
\label{sec:smoothn-some-algebr}

\begin{theorem}
  \label{t:DbmodA-smooth}
  Let $A$ be an algebra (not assumed to be commutative) over a
  perfect field
  $\kk$. Assume
  that $A$ is a finite module over its center $\Z(A)$
  and that the center $\Z(A)$ is a finitely generated
  (commutative) $\kk$-algebra.
  Then
  $\D^\bd(\mod(A))$ is smooth over $\kk$
  (in the sense defined in Remark~\ref{r:DbmodA-smooth}).
\end{theorem}

This result is a generalization of the version of
Theorem~\ref{t:Dbmod-findimalg-separable-smooth} where $\kk$ is a
perfect field. The strategy of proof is very similar.

\begin{proof}
  Remember that $\D^\bd(\mod(A))$ is equivalent to the category
  \begin{equation}
    \label{eq:T2}
    \mathcal{T}:=\D^\bd_{\mod(A)}(A)=\D^\bd(A)_\pseudocoh
  \end{equation}
  because $A$ is noetherian (see equivalence
  \eqref{eq:DbmodA-DbApscoh} and equality
  \eqref{eq:Db_modA-DbApscoh}
  in Remark~\ref{r:algebra-as-lincat}).
  By our definition
  of $\kk$-smoothness of $\D^\bd(\mod(A))$
  in
  Remark~\ref{r:DbmodA-smooth}
  we need to prove that $\mathcal{T}$ is $\kk$-smooth.
  We will use the sufficient condition for smoothness of
  Theorem~\ref{t:sufficient-for-smoothness}.

  We first need to find a dualizing bimodule for $\mathcal{T}$
  (in the sense of Definition~\ref{d:dualizing}).
  This will use the notion of a dualizing
  complex from commutative algebra, see
  \cite[\sptag{0A7A}]{stacks-project}.

  We abbreviate $R:=\Z(A)$. This is a finitely generated
  $\kk$-algebra by assumption.
  Hence $R$ has a dualizing complex $\omega$, by
  \cite[\sptag{0A7K}]{stacks-project}. In particular,
  \begin{equation}
    \label{eq:Hom-to-omega-adjunction}
    \Rd\Hom_R(-, \omega)
    \colon
    \D^\bd_{\mod(R)}(R)
    \rightleftarrows
    \D^\bd_{\mod(R)}(R)^\opp
    \colon
    \Rd\Hom_R(-, \omega)
  \end{equation}
  is an adjoint equivalence, by
  \cite[\sptag{0A7C}]{stacks-project}.
  Note that $\omega \in \D^\bd_{\mod(R)}(R)$.
  We can and will
  assume in the following that $\omega$ is a bounded below
  complex of injective $R$-modules and hence an
  h-injective complex of $R$-modules; this means that we can
  replace $\Rd\Hom$ by $\Hom$ in the above adjunction and assume
  that unit and counit of the adjunction are the obvious maps
  into the biduals with respect to $\omega$.

  If $M \in \C(A)$ is a complex of (right) $A$-modules, then $\Hom_R(A,
  \omega) \in \C(A^\opp)$ is a complex of left
  $A$-modules. Similarly, if $N \in \C(A^\opp)$ is a complex of
  left $A$-modules, then $\Hom_R(A,\omega) \in \C(A)$ is a
  complex of (right) $A$-modules. Moreover, the
  unit $M \ra \Hom_R(\Hom_R(M,\omega), \omega)$ and
  counit $N \ra \Hom_R(\Hom_R(N,\omega), \omega)$ are morphisms
  in $\C(A)$ and $\C(A^\opp)$, respectively. Hence the adjunction
  in the lower row of the following diagram gives rise to
  the adjunction in its upper row (note that the functors need
  not be decorated with a derived symbol).
  \begin{equation*}
    \xymatrix{
      {\D(A)}
      \ar@<0.5ex>[rr]^-{\Hom_R(-, \omega)}
      \ar[d]
      &&
      {\D(A^\opp)^\opp}
      \ar@<0.5ex>[ll]^-{\Hom_R(-, \omega)}
      \ar[d]
      \\
      {\D(R)}
      \ar@<0.5ex>[rr]^-{\Hom_R(-, \omega)}
      &&
      {\D(R)^\opp}
      \ar@<0.5ex>[ll]^-{\Hom_R(-, \omega)}
    }
  \end{equation*}
  The vertical arrows are the restriction functors along $R=\Z(A)
  \hra A$. The diagram is commutative if we ignore the two
  horizontal arrows
  pointing to the left or the two horizontal arrows pointing to
  the right.

  An $A$-module is finite over $A$ if and only if it is finite
  over $R$, because $A$ is a finite $R$-module. Hence
  an object $M \in \D(A)$ is in $\D^\bd_{\mod(A)}(A)$
  if and only if its restriction $M|_R \in \D(R)$
  is in $\D^\bd_{\mod(R)}(R)$.
  Since we know that the lower adjunction restricts to the
  adjoint equivalence
  \eqref{eq:Hom-to-omega-adjunction}, we deduce that the upper
  adjunction restricts to an adjoint equivalence
  \begin{equation}
    \label{eq:Hom-to-omega-adjunction-for-A}
    \Hom_R(-, \omega)
    \colon
    \mathcal{T}=\D^\bd_{\mod(A)}(A)
    \rightleftarrows
    \D^\bd_{\mod(A^\opp)}(A^\opp)^\opp
    \colon
    \Hom_R(-, \omega).
  \end{equation}
  We will see that this adjoint equivalence originates from a
  complex of bimodules. The natural candidate is
  \begin{equation*}
    \mathscr{D}:=\Hom_R(A, \omega)=\Hom_R(\leftidx{_A}{A}{_A}, \omega)
  \end{equation*}
  which is a complex of $A \otimes_R A^\opp$-modules and may also
  be viewed as a complex of $A \otimes_\kk A^\opp$-modules.
  We have
  \begin{multline*}
    \Hom_A(M, \mathscr{D})
    =\Hom_A(M,\Hom_R(\leftidx{_A}{A}{_A}, \omega))\\
    =\Hom_R(M \otimes_A \leftidx{_A}{A}{_A}, \omega)
    =\Hom_R(M,\omega)
  \end{multline*}
  in $\C(A^\opp)$ natural in $M \in \C(A)$
  and
  \begin{multline*}
    \Hom_{A^\opp}(N, \mathscr{D})
    =\Hom_{A^\opp}(N,\Hom_R(\leftidx{_A}{A}{_A}, \omega))\\
    =\Hom_R(\leftidx{_A}{A}{_A} \otimes_A N, \omega)
    =\Hom_R(N, \omega)
  \end{multline*}
  in $\C(A)$ natural in $N \in \C(A^\opp)$.
  Hence the two functors
  in the adjoint equivalence
  \eqref{eq:Hom-to-omega-adjunction-for-A}
  may be written as
  \begin{equation*}
    \Hom_R(-, \omega) = \Hom_A(-,\mathscr{D})
    \qquad \text{and} \qquad
     \Hom_R(-, \omega)= \Hom_{A^\opp}(-,\mathscr{D}).
  \end{equation*}
  Moreover, the unit and counit of the adjunction
  \eqref{eq:Hom-to-omega-adjunction-for-A}
  correspond to the unit and counit
  of the adjunction obtained from $\mathscr{D}$, cf.\
  \eqref{eq:Hom-to-X-adjunction}.
  Since \eqref{eq:Hom-to-omega-adjunction-for-A}
  is an adjoint equivalence,
  Remark~\ref{r:dualizing-object-from-duality} shows that
  $\mathscr{D}$ is a dualizing object for $\mathcal{T}$ and that
  \begin{equation}
    \label{eq:Tvee2}
    \mathcal{T}^\vee=\D^\bd_{\mod(A^\opp)}(A^\opp)=\D^\bd(A^\opp)_\pseudocoh
  \end{equation}
  where the last equality comes from \eqref{eq:Db_modA-DbApscoh}.

  It remains to check conditions
  \ref{enum:T-class-gen}
  and
  \ref{enum:dualizing-object-from-T-Tvee}
  from
  Theorem~\ref{t:sufficient-for-smoothness}
  in our situation.

  Let $X=\Spec R$ and
  let $\mathcal{A}$ be the coherent $\mathcal{O}_X$-algebra
  corresponding to the finite $R$-algebra $A$.
  Then we have equivalences of categories
  \begin{equation*}
    \mathcal{T}=\D^\bd_{\mod(A)}(A) \cong \D^\bd(\mod(A)) \cong
    \D^\bd(\coh(\mathcal{A}))
  \end{equation*}
  (cf.\ the equivalence \eqref{eq:cohA-modA}).
  Since $X=\Spec R=\Spec \Z(A)$ is of finite type over $\kk$ it is
  a noetherian J-2 scheme, and hence
  Theorem~\ref{t:generator-DbcohA}
  shows that $\D^\bd(\coh(\mathcal{A}))$ and hence $\mathcal{T}$
  have a classical generator.
  This together with the equalities
  \eqref{eq:T2} and
  \eqref{eq:Tvee2} shows that
  condition~\ref{enum:T-class-gen} is satisfied.

  Let $E$ be a classical generator of $\mathcal{T}
  \cong \D^\bd(\coh(\mathcal{A}))$. Its
  dual $F:=\Hom_A(E, \mathscr{D})$ is then a classical generator of
  $\mathcal{T}^{\vee} \cong \D^\bd(\coh(\mathcal{A}^\opp))$.
  In order to check
  condition~\ref{enum:dualizing-object-from-T-Tvee}, or rather
  the equivalent condition
  \ref{enum:dualizing-object-from-class-gen-T-Tvee}
  in Remark~\ref{r:on-t:sufficient-for-smoothness}, we need to
  prove that the thick subcategory
  of $\D(A \otimes_\kk A^\opp)$
  generated by $E \otimes_\kk F$ contains $\mathscr{D}$.
  Obviously we have
  \begin{equation*}
    \mathscr{D} =\Hom_R(A, \omega)
    \in
    \D^\bd_{\mod(A \otimes_\kk A^\opp)}(A \otimes_\kk A^\opp)
    \cong
    \D^\bd(\coh(\mathcal{A} \boxtimes \mathcal{A}^\opp))
  \end{equation*}
  where $\mathcal{A} \boxtimes \mathcal{A}^\opp$ is the
  coherent $\mathcal{O}_{X \times_\kk X}$-algebra corresponding
  to the finite $(R \otimes_\kk R)$-algebra $A \otimes_\kk A^\opp$.
  If we view $E$ as an object of $\D^\bd(\coh(\mathcal{A}))$ and
  $F$ as an object of $\D^\bd(\coh(\mathcal{A}^\opp))$ it is
  therefore enough to show that $E \boxtimes F$ is a classical
  generator of
  $\D^\bd(\coh(\mathcal{A} \boxtimes \mathcal{A}^\opp))$.
  But this is true by Theorem~\ref{t:boxtimes-and-generators}
  since $\kk$ is assumed to be perfect.
\end{proof}

\begin{remark}
  \label{r:finite-global-dim-smooth-2}
  This is the analog to Remark~\ref{r:finite-global-dim-smooth}.
  Let $A$ be an algebra over a perfect field $\kk$ as in
  Theorem~\ref{t:DbmodA-smooth}, i.\,e.\
  $A$ is a finite module over its center $\Z(A)$
  and the center $\Z(A)$ is a finitely generated
  $\kk$-algebra.
  Assume in addition that $A$ is
  right-regular in sense of \cite[7.7.1]{McCRob}: Any finitely
  generated module has finite projective dimension.
  Then $A$ is a classical generator of
  $\D^\bd(\mod(A))$. Hence
  $A$ is $\kk$-smooth, by
  Theorem~\ref{t:DbmodA-smooth} and
  Remark~\ref{r:smoothness-dg-endos-classical-generator}.
\end{remark}

\def\cprime{$'$} \def\cprime{$'$} \def\cprime{$'$} \def\cprime{$'$}
  \def\Dbar{\leavevmode\lower.6ex\hbox to 0pt{\hskip-.23ex \accent"16\hss}D}
  \def\cprime{$'$} \def\cprime{$'$}


\begin{thebibliography}{ARS97}

\bibitem[AG60]{auslander-goldman-brauer}
Maurice Auslander and Oscar Goldman.
\newblock The {B}rauer group of a commutative ring.
\newblock {\em Trans. Amer. Math. Soc.}, 97:367--409, 1960.

\bibitem[AM69]{atiyah-macdonald}
M.~F. Atiyah and I.~G. Macdonald.
\newblock {\em Introduction to commutative algebra}.
\newblock Addison-Wesley Publishing Co., Reading, Mass.-London-Don Mills, Ont.,
  1969.

\bibitem[ARS97]{ARS-rep-artin-algebras}
Maurice Auslander, Idun Reiten, and Sverre~O. Smal\o.
\newblock {\em Representation theory of {A}rtin algebras}, volume~36 of {\em
  Cambridge Studies in Advanced Mathematics}.
\newblock Cambridge University Press, Cambridge, 1997.
\newblock Corrected reprint of the 1995 original.

\bibitem[BLS16]{daniel-valery-olaf-geometricity}
Daniel Bergh, Valery~A. Lunts, and Olaf~M. Schn\"{u}rer.
\newblock Geometricity for derived categories of algebraic stacks.
\newblock {\em Selecta Math. (N.S.)}, 22(4):2535--2568, 2016.

\bibitem[CS18]{canonaco-stellari-uniqueness-of-dg-enhancements}
Alberto Canonaco and Paolo Stellari.
\newblock Uniqueness of dg enhancements for the derived category of a
  {G}rothendieck category.
\newblock {\em J. Eur. Math. Soc. (JEMS)}, 20(11):2607--2641, 2018.

\bibitem[FD93]{farb-dennis-NC-algebra}
Benson Farb and R.~Keith Dennis.
\newblock {\em Noncommutative algebra}, volume 144 of {\em Graduate Texts in
  Mathematics}.
\newblock Springer-Verlag, New York, 1993.

\bibitem[GW10]{goertz-wedhorn-AGI}
Ulrich G{\"o}rtz and Torsten Wedhorn.
\newblock {\em Algebraic geometry {I}}.
\newblock Advanced Lectures in Mathematics. Vieweg + Teubner, Wiesbaden, 2010.

\bibitem[IT18]{iyengar-takahashi-openness-2018}
Srikanth~B. Iyengar and Ryo Takahashi.
\newblock Openness of the regular locus and generators for module categories.
\newblock {\em Acta Mathematica Vietnamica}, Sep 2018.

\bibitem[Iya14]{iyama-oberwolfach}
Osamu Iyama.
\newblock Conversation with {V}. {L}unts und {O}. {S}chn\"u{}rer in
  {O}berwolfach in {M}ay, 2014.

\bibitem[KO74]{knus-ojanguren-descente-azumaya}
Max-Albert Knus and Manuel Ojanguren.
\newblock {\em Th\'eorie de la descente et alg\`ebres d'{A}zumaya}.
\newblock Lecture Notes in Mathematics, Vol. 389. Springer-Verlag, Berlin-New
  York, 1974.

\bibitem[LC07]{le-chen-karoubi-trcat-bdd-t-str}
Jue Le and Xiao-Wu Chen.
\newblock Karoubianness of a triangulated category.
\newblock {\em J. Algebra}, 310(1):452--457, 2007.

\bibitem[LO10]{lunts-orlov-enhancement}
Valery~A. Lunts and Dmitri~O. Orlov.
\newblock Uniqueness of enhancement for triangulated categories.
\newblock {\em J. Amer. Math. Soc.}, 23(3):853--908, 2010.

\bibitem[LS16a]{valery-olaf-matrix-factorizations-and-motivic-measures}
Valery~A. Lunts and Olaf~M. Schn\"{u}rer.
\newblock Matrix factorizations and motivic measures.
\newblock {\em J. Noncommut. Geom.}, 10(3):981--1042, 2016.

\bibitem[LS16b]{valery-olaf-new-enhancements}
Valery~A. Lunts and Olaf~M. Schn{\"u}rer.
\newblock New enhancements of derived categories of coherent sheaves and
  applications.
\newblock {\em J. Algebra}, 446:203--274, 2016.

\bibitem[Lun10]{lunts-categorical-resolution}
Valery~A. Lunts.
\newblock Categorical resolution of singularities.
\newblock {\em J. Algebra}, 323(10):2977--3003, 2010.

\bibitem[MR87]{McCRob}
J.~C. McConnell and J.~C. Robson.
\newblock {\em Noncommutative {N}oetherian rings}.
\newblock Pure and Applied Mathematics (New York). John Wiley \& Sons Ltd.,
  Chichester, 1987.
\newblock With the cooperation of L. W. Small, A Wiley-Interscience
  Publication.

\bibitem[Nee17]{neeman-strong-generators}
Amnon Neeman.
\newblock Strong generators in {$D^{perf}(X)$} and {$D^b_{coh}(X)$}.
\newblock {\em Preprint}, 2017.
\newblock \href{http://arxiv.org/abs/1703.04484}{arxiv:1703.04484v2}.

\bibitem[Rou08]{rouquier-dimensions}
Rapha{\"e}l Rouquier.
\newblock Dimensions of triangulated categories.
\newblock {\em J. K-Theory}, 1(2):193--256, 2008.

\bibitem[SGA6]{berthelot-grothendieck-illusie-SGA-6}
P.~Berthelot, A.~Grothendieck, and L.~Illusie.
\newblock {\em Th\'eorie des intersections et th\'eor\`eme de
  {R}iemann-{R}och}.
\newblock Lecture Notes in Mathematics, Vol. 225. Springer-Verlag, Berlin,
  1971.
\newblock S{\'e}minaire de G{\'e}om{\'e}trie Alg{\'e}brique du Bois-Marie
  1966--1967 (SGA 6).

\bibitem[SP18]{stacks-project}
The {Stacks Project Authors}.
\newblock The {S}tacks {P}roject.
\newblock \url{http://stacks.math.columbia.edu}, 2018.

\bibitem[TT90]{thomason-trobaugh-higher-K-theory}
R.~W. Thomason and Thomas Trobaugh.
\newblock Higher algebraic {$K$}-theory of schemes and of derived categories.
\newblock In {\em The {G}rothendieck {F}estschrift, {V}ol.\ {III}}, volume~88
  of {\em Progr. Math.}, pages 247--435. Birkh\"auser Boston, Boston, MA, 1990.

\bibitem[TY03]{tousi-yassemi-tensor-product-of-rings}
Masoud Tousi and Siamak Yassemi.
\newblock Tensor products of some special rings.
\newblock {\em J. Algebra}, 268(2):672--676, 2003.

\bibitem[Ver96]{verdier-these}
Jean-Louis Verdier.
\newblock Des cat\'egories d\'eriv\'ees des cat\'egories ab\'eliennes.
\newblock {\em Ast\'erisque}, (239):xii+253 pp. (1997), 1996.
\newblock With a preface by Luc Illusie, Edited and with a note by Georges
  Maltsiniotis.

\end{thebibliography}

\end{document}